\documentclass{article}

\usepackage[letterpaper,top=2cm,bottom=3cm,left=3cm,right=3cm,marginparwidth=2.2cm]{geometry}

\usepackage{amsmath,amsfonts,amssymb,amsthm}
\usepackage{color}
\usepackage{graphicx}
\newcommand{\eps}{\varepsilon}
\usepackage[dvipsnames]{xcolor}
\usepackage{authblk}
\usepackage[english]{babel}
\usepackage{hyperref}

\newtheorem{theorem}{Theorem}[section]
\newtheorem{lemma}[theorem]{Lemma}

\newtheorem{remark}{Remark}

\title{The Lingering Phenomenon and Pattern Formation in a Nonlocal Population Model with Cognitive Map
}
\author[1]{Kyung-Han Choi\thanks{Corresponding author. Email: kyunghan@ualberta.ca}}
\author[1]{Thomas Hillen}
\affil[1]{Department of Mathematical and Statistical Sciences, 
University of Alberta, Edmonton, Alberta T6G 2G1, Canada}

\date{}
\begin{document}
\maketitle
\begin{abstract}
The rates at which individuals memorize and forget environmental information strongly influence their movement paths and long-term space use. 
To understand how these cognitive time scales shape population-level patterns, we propose and analyze a nonlocal population model with a cognitive map. 
The population density moves by a Fokker--Planck type diffusion driven by a cognitive map that stores a habitat quality information nonlocally. 
The map is updated through local presence with learning and forgetting rates, and we consider both truncated and normalized perception kernels.

We first study the movement-only system without growth. 
We show that finite perceptual range generates spatial heterogeneity in the cognitive map even in nearly homogeneous habitats, and we prove a lingering phenomenon on unimodal landscapes: for the fixed high learning rate, the peak density near the best location is maximized at an intermediate forgetting rate.

We then couple cognitive diffusion to logistic growth. 
We establish local well-posedness and persistence by proving instability of the extinction equilibrium and the existence of a positive steady state, with uniqueness under an explicit condition on the motility function.  
Numerical simulations show that lingering persists under logistic growth and reveal a trade-off between the lingering and total population size, since near the strongest-lingering regime the total mass can fall below the total resource, in contrast to classical random diffusive--logistic models.
\end{abstract}

\begin{flushleft}
\end{flushleft}

\renewcommand{\thefootnote}{\fnsymbol{footnote}}
\tableofcontents

\section{Introduction}

The rates at which individuals memorize and forget new environmental information are closely linked to their spatial navigation and resulting patterns of space use. In neuroscience, studies have revealed a relationship between hippocampal neurogenesis and roaming patterns via \textit{roaming entropy}, a measure of active territory coverage \cite{Freund2013,Heller2020}. A recent study found that patients with Alzheimer’s disease show significantly different roaming entropy compared with healthy older adults \cite{Ghosh2022}. In particular, \cite{BRUNEC2023105360} links spatial navigation and the formation of cognitive maps through roaming entropy, highlighting the importance of accounting for both individual memory variability and the structure of the environment.

Although neuroscientific insights into learning, forgetting, and roaming have rarely been extended to population-level ecology, recent mathematical models formalize spatial memory using cognitive map frameworks \cite{Painter2024review,Potts2016Royal,Potts2016,Potts2019} and examine its consequences for home range formation and aggregation \cite{Liu2025,Shi2021,Xue2024}. In addition, \cite{Wang2023} outlines open modeling and analytical challenges for cognitive movement models. In many of these studies, cognitive movement is represented as diffusion and advection with a nonlocal advection term, where a nonlocal memory field biases movement. This nonlocality is thought to be biologically natural because organisms sense their surroundings beyond a single point through vision or smell, so models introduce various types of spatial kernels with a finite perceptual radius, such as the top-hat kernel and bump functions. As a result, much of the literature investigates how the kernel, especially its radius and shape, affects aggregation, home range size, and spatial patterning \cite{Fagan2017,Giunta2024,Liu2025,Painter2024chase}. Those models have been studied under periodic boundary condition on the domain, which makes it difficult to fully capture how habitat boundaries influence the information animals perceive.

We hypothesize that memory-based navigation with nonlocal perception can generate heterogeneous patterns of space use and residence times even in nearly homogeneous landscapes. For example, even when food availability is uniform, limited visual range and habitat structure, such as obstacles or cover, create heterogeneity in the cognitive map, consistent with evidence that environmental geometry shapes cognitive maps \cite{BRUNEC2023105360}.  On such a heterogeneous cognitive map, low-retention environmental cues are forgotten before they can guide movement, shifting use toward strongly remembered locations; therefore the rates of learning and forgetting are critical. These ideas lead to a central question about how learning and forgetting rates shape long-term population distributions in ecological systems. It is expected that insufficient memory yields diffusion-like wandering and poor patch use, whereas excessive memory yields maladaptive persistence to stale information. We also expect that there exists a moderate balance that maximizes residency near favorable regions. We refer to this effect as \emph{lingering}.
\begin{figure}[ht]
    \centering
    \includegraphics[width=0.7\linewidth]{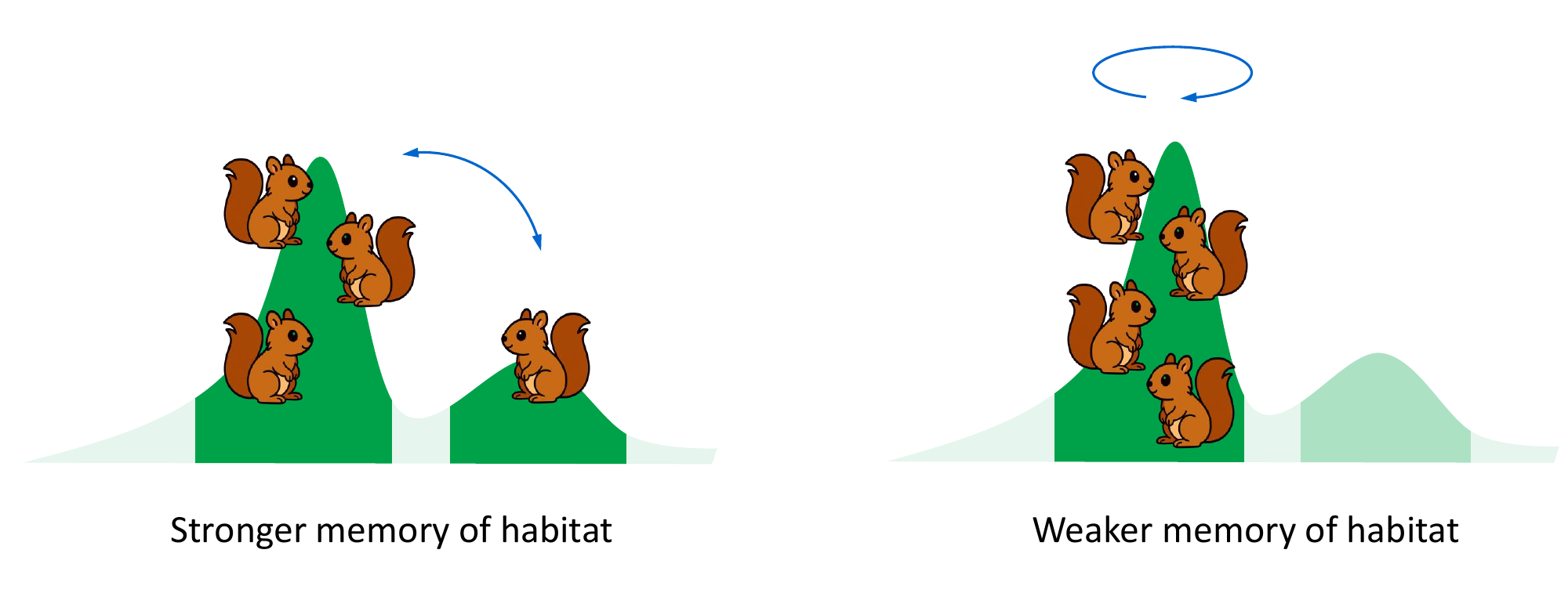}
    \caption{Schematic illustration of lingering in a one-dimensional habitat. The green curve represents perceived habitat quality (a cognitive map), with darker shading indicating locations that are more strongly remembered. Squirrels depict where individuals tend to reside. In the left panel, animals occupy two high-memory-retention regions and possibly move between the peaks. In the right panel, individuals linger in a single strongly remembered core region, leaving other peak areas underused.}
    \label{fig:lingering_squirrel}
\end{figure}

\subsection{The Mathematical Model}

We formalize these ideas by coupling population density $u(x,t)$ on a bounded domain $\Omega \subset \mathbb{R}^n$ with a cognitive map $m(x,t)$ that encodes spatial memory of habitat density $s(x)$. We work with a dimensionless formulation. The quantity $u(x,t)$ denotes a scaled population density and $s(x)$ represents the corresponding scaled habitat quality or carrying capacity, so that $(s(x)-u)u$ has the form of a logistic growth term with spatially varying capacity. The nonlocal quantity $\bar s(x)$ represents the perceived habitat quality obtained by spatially averaging $s$ around $x$. Thus $s$, $\bar s$ and $m$ are all dimensionless and measured on the same scale. The movement and memory dynamics are
\begin{equation}\label{eqn:main}
\left\{
\begin{aligned}
&\partial_t u  = \Delta\big(\gamma(m) u\big)+ (s(x)-u)\,u
&& \text{in } \Omega\times(0,T], \\[4pt]
&\partial_t m = \big(\alpha\,\bar s(x)-m\big)\,u - \mu\,m
&& \text{in } \Omega\times(0,T],\\
&\nabla\big(\gamma(m) u\big)\cdot \mathbf{n} = 0
&& \text{on } \partial\Omega\times(0,T],\\
&u(x,0) = u_0(x), \quad m(x,0) = m_0(x)
&& \text{in } \Omega,
\end{aligned}
\right.
\end{equation}
where $\gamma(m)>0$ renders memory-dependent mobility, $\alpha>0$ represents the rate at which individuals memorize (or learning strength), and $\mu>0$ is the memory decay rate, that is, how quickly they forget their memory.
The update term $(\alpha\bar{s}-m)u$ implements ``learn--where--you--are’’: memory is reinforced (or down–weighted) proportionally to local presence $u$, depending on the perceived favorability $\alpha\bar{s}$. The memory equation for $m$ provides a mechanistic way to construct a cognitive map. The similar dynamics for cognitive map is demonstrated in literature \cite{Liu2025, Wang2023}. We define the nonlocal quantity
\begin{equation}\label{bars}
\bar{s}(x)=(K*s)(x)=\int_{B_R(x)}K(x,y)\,s(y)\chi_{\overline{\Omega}}(y)\,dy 
\end{equation}
with
\begin{equation}\label{defkernel}K(x,y) = \left\{\begin{aligned}
    &J_R(x-y) \text{ or }\frac{J_R(x-y)}{Z_R(x)} &&\text{ if } x-y\in B_R(0),\\
    &\qquad\qquad 0&& \text{ otherwise, }
\end{aligned}\right. 
\end{equation}
where $Z_R(x)=\int_{B_R(x)} J_R(x-y)\chi_{\bar\Omega}(y)dy$ and $J_R$ encodes a perception kernel whose support is a ball $B_R(0)$ centered at $0$ with radius $R$. $\chi_{\bar\Omega}$ is a characteristic function, that is, it is equal to 1 in $\overline\Omega$ and otherwise it is 0, and we illustrate its use in Figure \ref{fig:weirdOmega}. We will consider a normalized and a not-normalized version of the integral operator in (\ref{defkernel}). 

\begin{figure}[ht]
    \centering
    \includegraphics[width=0.4\linewidth]{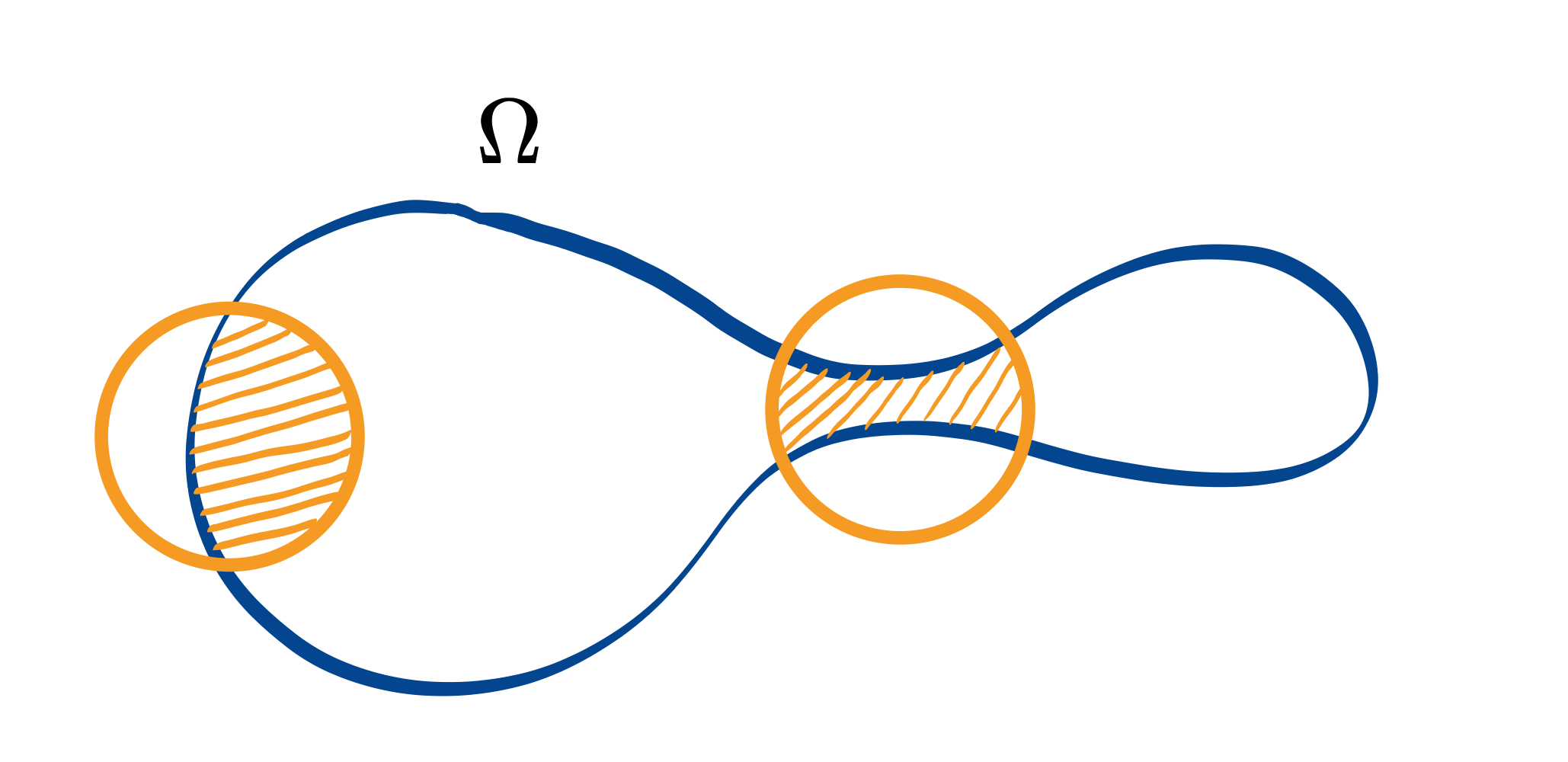}
    \caption{The blue curve represents the boundary of the habitat $\Omega$ occupied by the population. Dashed regions represent the nonlocal neighborhoods perceived under a finite perceptual radius at different locations.}
    \label{fig:weirdOmega}
\end{figure}

We assume that $\gamma$ is a positive, decreasing function of $m$. In other words, individuals move more slowly in regions where the memory variable $m$ is high, tending to remain longer in locations they remember as favorable. Formally expanding the diffusion term generates a natural advection term induced by gradients of $m$, and since $\gamma'<0$, this advection describes attraction toward regions of higher memory. The expansion of the diffusion term is \begin{equation*}
    \Delta(\gamma(m)u) = \nabla\cdot\left(\gamma(m)\nabla u + \gamma'(m)u\nabla m\right).
\end{equation*}
This type of diffusion operator, $\Delta(\gamma(m)u)$, is often referred to as a Fokker--Planck--type diffusion or a Chapman diffusion law and has been extensively used to model density-dependent dispersal on the heterogeneous media in biology and ecology, where the probability of leaving a location is determined by the local state (for example, the local density) rather than by its gradient \cite{Chang2025, Cho2013, Kim2023chemotaxis, Kim2014, Tang2023}. In our setting, this motility depends on the memory field $m$ rather than directly on the density $s$, but the underlying diffusion structure is of the same Fokker--Planck type.

It is noteworthy to mention that nonlocality enters only through the second equation in \eqref{eqn:main}. Much of the existing literature introduces nonlocal space use via an advection term, which has typically restricted the analysis to periodic boundary conditions \cite{Liu2025, Shi2021,Xue2024}. In contrast, our use of no-flux boundary conditions together with the domain-restricted kernel $J_R(x-y)\chi_{\bar \Omega} (y) $ allows us to clarify how the geometric structure of a bounded domain shapes perceived habitat information. As illustrated in Figure~\ref{fig:weirdOmega}, the finite perceptual radius generates nonlocal neighborhoods near the boundary that differ from those in the interior, so the geometry of $\Omega$ induces diverse types of spatial heterogeneity in the perceived environment.

On the other hand, we also consider the normalized kernel $K(x,y)=\frac{J_R(x-y)\chi_{\bar \Omega} (y)}{Z_R(x)}$ to model situations in which individuals are only able to sense a limited amount of environmental cues. In this case, the normalization compensates for boundary effects, and the influence of the domain geometry on the perceived habitat information is suppressed.

\subsection{Paper outline}
In Section~\ref{sec:lingering from dispersal} we analyze a cognitive movement model without population growth and study spatial patterns arising from nonlocal perception. 
We identify the lingering phenomenon on heterogeneous cognitive maps using both mathematical analysis and numerical simulations.
In Section~\ref{sec:population dynamics} we turn to the full model \eqref{eqn:main} with logistic growth, establish its well-posedness, and investigate its long-time behavior. 
We also show numerically that lingering persists despite logistic-type growth and examine how the learning and forgetting rates shape not only the peak density at favorable locations but also the total population size. 
Finally, in Section~\ref{sec:discussion} we summarize our findings, discuss the limitations of the study, and outline directions for future work.

\section{Lingering phenomenon on heterogeneous cognitive map}\label{sec:lingering from dispersal}

In this section, we investigate how spatial heterogeneity in the cognitive map, generated by nonlocal perception, gives rise to a lingering phenomenon in space use. In Section~2.1, we construct heterogeneous cognitive maps \(\bar s(x)\) via nonlocal perception and illustrate their effects numerically. In Section~2.2, we then develop a mathematical theory that characterizes the lingering phenomenon on heterogeneous cognitive maps. To isolate the effect of the population growth, we consider movement driven solely by context-dependent diffusion. We consider the cognitive diffusion model:
\begin{equation}\label{eqn: dispersalonly}
\left\{
\begin{aligned}
&\partial_t u  = \Delta\big(\gamma(m) u\big)
&& \text{in } \Omega\times(0,T], \\[4pt]
&\partial_t m = \big(\alpha\,\bar s(x)-m\big)\,u - \mu\,m
&& \text{in } \Omega\times(0,T]\\
&\nabla \big(\gamma(m) u\big)\cdot \mathbf{n} = 0
&& \text{on } \partial\Omega\times(0,T],
\end{aligned}
\right.
\end{equation} where \(\gamma:[0,\infty)\to(0,\infty)\) is a positive, strictly decreasing motility function. The perceived quantity \(\bar s(x)\) is defined on \(\Omega\) as in~\eqref{defkernel} and is independent of time in this section. Throughout the simulations below, we choose
\begin{equation}\label{gamma}
\gamma(z) = \frac{1}{(1+z)^2}, \qquad z \ge 0.
\end{equation}

\subsection{Heterogeneity by nonlocal perception}
\hspace{.1cm}
This subsection demonstrates that non-normalized kernels introduce artificial heterogeneity near the boundary, even in a homogeneous environment, whereas normalized kernels eliminate this effect.
This is because the range in which the species perceives $s(y)$ gets smaller due to the lack of information outside of the domain. 

Let us consider a homogeneous environment, that is, $s(x)\equiv s$ for all $x\in\overline{\Omega}$.
Let $\Omega\subset\mathbb{R}^d$ be bounded and $J_R\ge 0$ be an integrable kernel with compact support contained in $B_R(0)$ and 
$I_R:=\int_{B_R(0)} J_R(r)\,dr<\infty$. For $x\in\overline\Omega$, 
\[ \bar s(x) = \int_{B_R(x)}J_R(x-y)s\chi_{\overline{\Omega}}(y)dy=sZ_R(x),\] where $Z_R(x):=\int_{B_R(x)} J_R(x-y)\,\chi_{\overline\Omega}(y)\,dy$. In this case, the non-normalized kernel corresponds to collecting \emph{total amount of information} within the perceptual radius.
Since $0\le\chi_{\overline\Omega}\le1$ and $B_R(x)\cap\Omega\subset B_R(x)$,
\[
Z_R(x)=\int_{B_R(x)\cap\overline\Omega} J_R(x-y)\,dy
\le\int_{B_R(x)} J_R(x-y)\,dy=I_R.
\] 
Strict inequality holds if $B_R(x)\setminus\overline\Omega$ has positive measure and $J_R\not\equiv0$ there. Therefore, boundary truncation directly lowers the perceived quantity.

Assume that $\Omega=(-5,5)$. We fix the memory-related parameters $\alpha=2$ and $\mu=0$, so we look at the case when species utilize the perceived quantity fully to move. The qualitative effect of boundary truncation is similar when $\mu>0$. We use a smooth bump kernel for illustration: 
\begin{equation}\label{bump}
    J_R(r) = \left\{\begin{aligned}
        &e^{-1/(1-r^2/R^2)} && \text{ if } |r|<R,\\
        &\,\,0 && \text{ otherwise.}
    \end{aligned}\right.
\end{equation}
We show the kernel $J_R$ for several values of $R$ in Figure \ref{fig: bars and u with nonscaled} on the left. 
Qualitatively similar truncation arises for other compactly supported kernels such as top-hat kernels.
\begin{figure}[ht]
    \centering
    \includegraphics[width=0.33\linewidth]{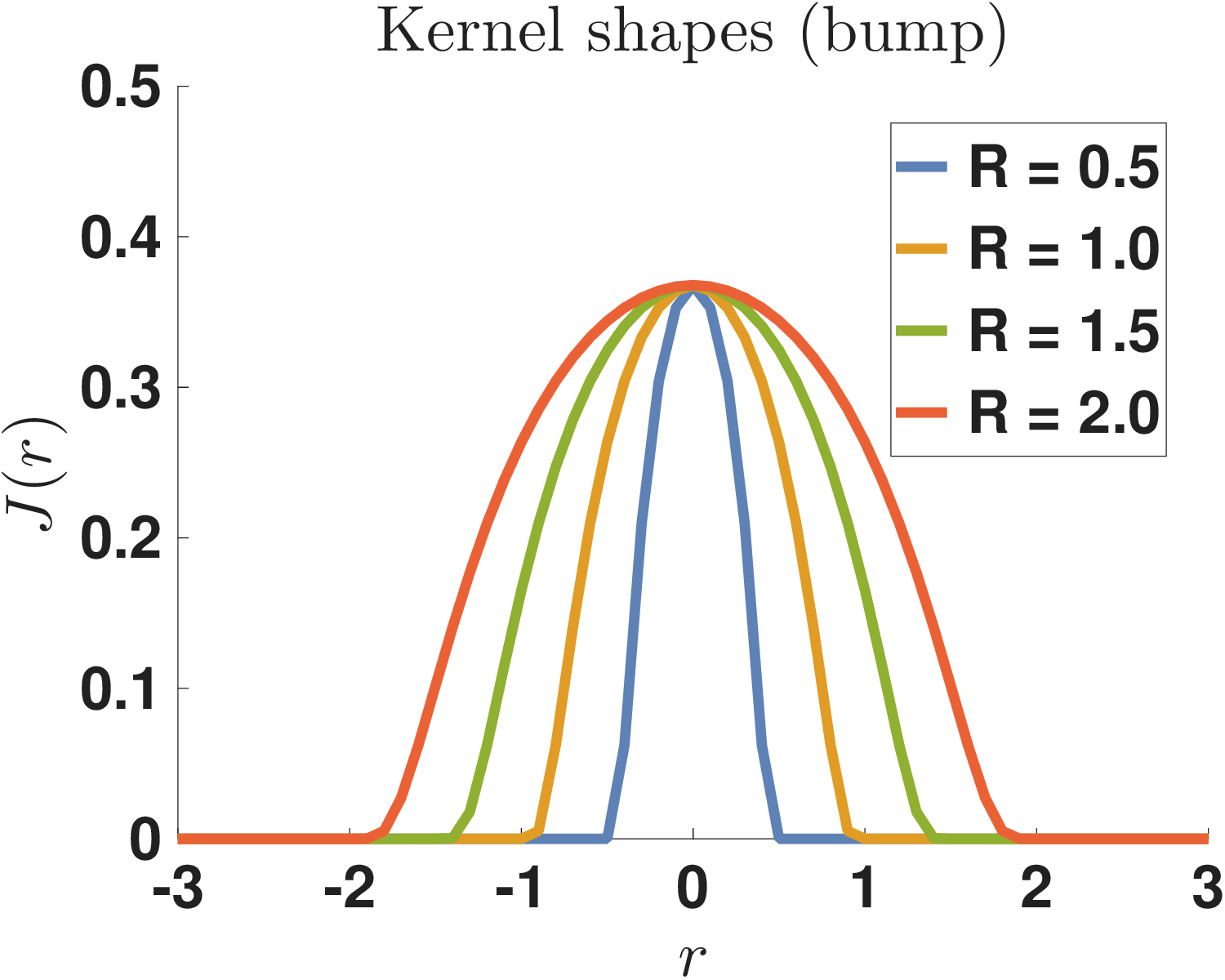}\includegraphics[width=0.33\linewidth]{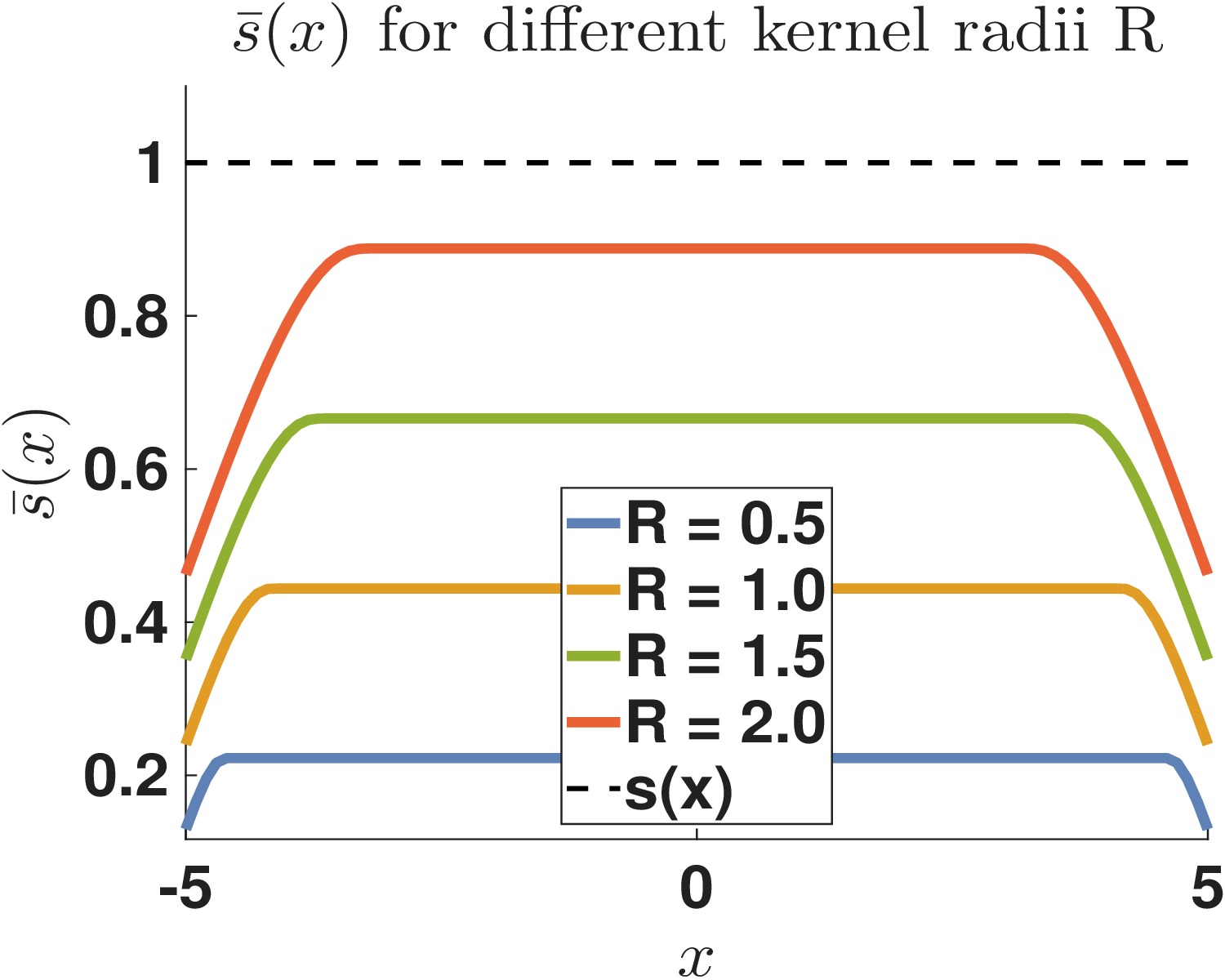}~
    \includegraphics[width=0.33\linewidth]{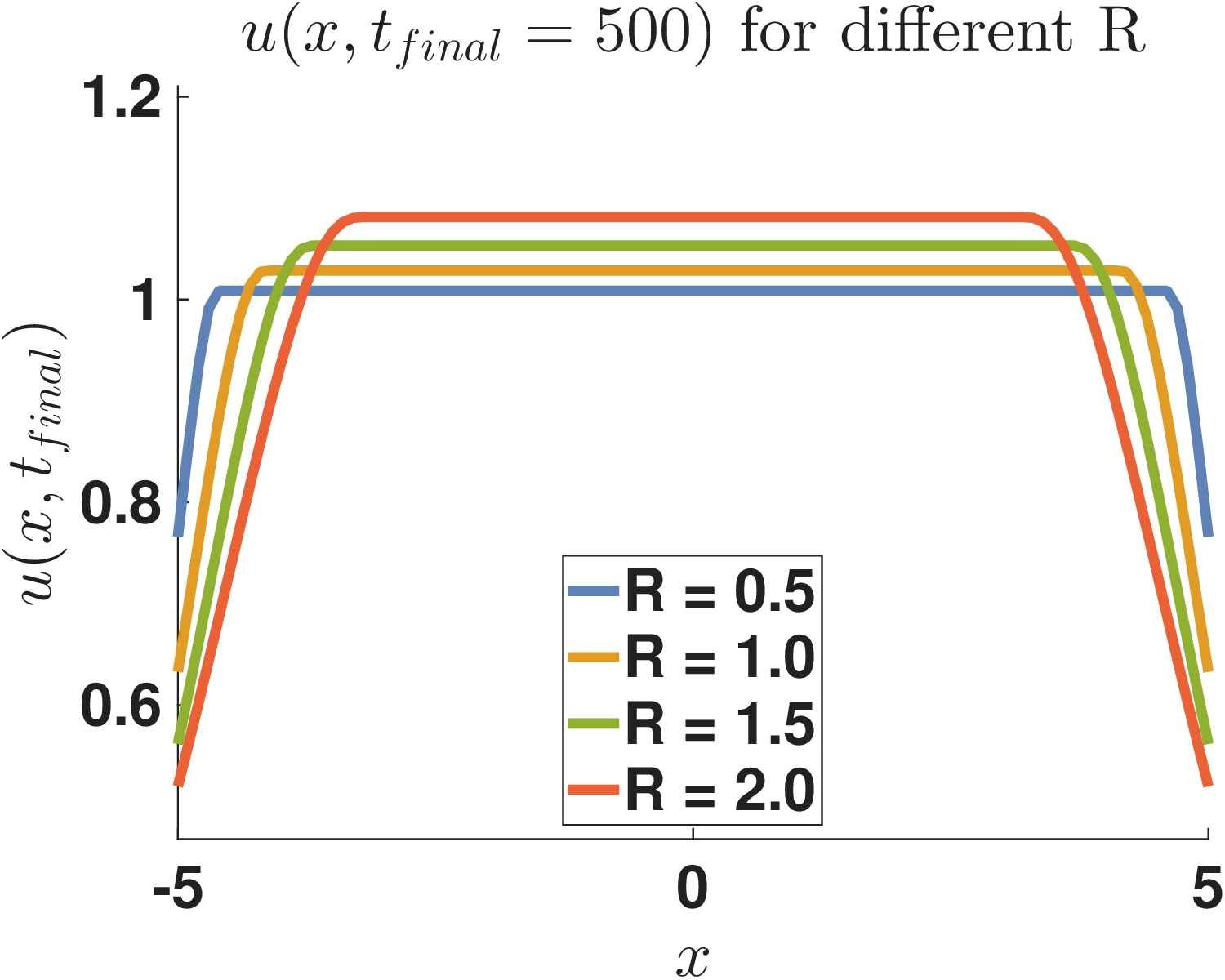}
    \caption{ The left first figure is different kernel shapes of $J_R$ for radius $R=0.5,1,1.5$ and 2. The second plot illustrates $\bar s$ corresponding to kernels in the left figure. The black dotted line is $s(x)\equiv 1$. $\bar s(x)\approx s$ in the interior but decreases near the boundary, more strongly for larger $R$. The figure on the right depicts the numerical solution $u(x,t_{\text{final}})$ at a large time $t=500$. As $R$ increases, $u(x,t_{\text{final}})$ becomes more concentrated in the middle due to the artificial heterogeneity.}
    \label{fig: bars and u with nonscaled}
\end{figure} 
The non-normalized kernel $J_R$ allows individuals to gain different quantities of information near the boundary, which creates spatial heterogeneity on their cognitive map $m$. Due to the heterogeneity of the cognitive map, we observe nonuniform solutions at large time. (The third figure of Figure \ref{fig: bars and u with nonscaled}). We also observed that the population $u(x,500)$ is more focused in the middle of the domain as $R$ gets large. It is expected that the nonlocal quantity $\bar s$ gets more complicated when the boundary of the domain is more complex.

Now we explore the case of normalized kernels as shown in Figure \ref{fig:hetero s} on the right.
Recall the normalized nonlocal environment cue:
\begin{equation*}
    \bar s(x) = \int_{B_R(x)}\frac{ J_R(x-y)}{Z_R(x)}s(y)\chi_{\overline\Omega}(y)dy \text{ and } Z_R(x)=\int_{B_R(x) } J_R(x-y)\chi_{\overline\Omega}(y) dy.
\end{equation*} 
If $s(x)\equiv s$ for all $x\in \overline{\Omega}$, $\bar s(x)\equiv s$ for all $x\in\overline\Omega$. Indeed,
\[
\bar s(x)=\frac{s}{Z_R(x)}\int_{B_R(x)} J_R(x-y)\,\chi_\Omega(y)\,dy
=\frac{s}{Z_R(x)}\,Z_R(x)=s.
\] In this case, the normalized kernel corresponds to sensing \emph{average concentration} in the neighborhood.

The effect of nonlocal perception with the normalized kernel is prominent on the heterogeneous environment $s$. See Figure \ref{fig:hetero s}.

\begin{figure}[ht]
    \centering
    \includegraphics[width=0.3333\linewidth]{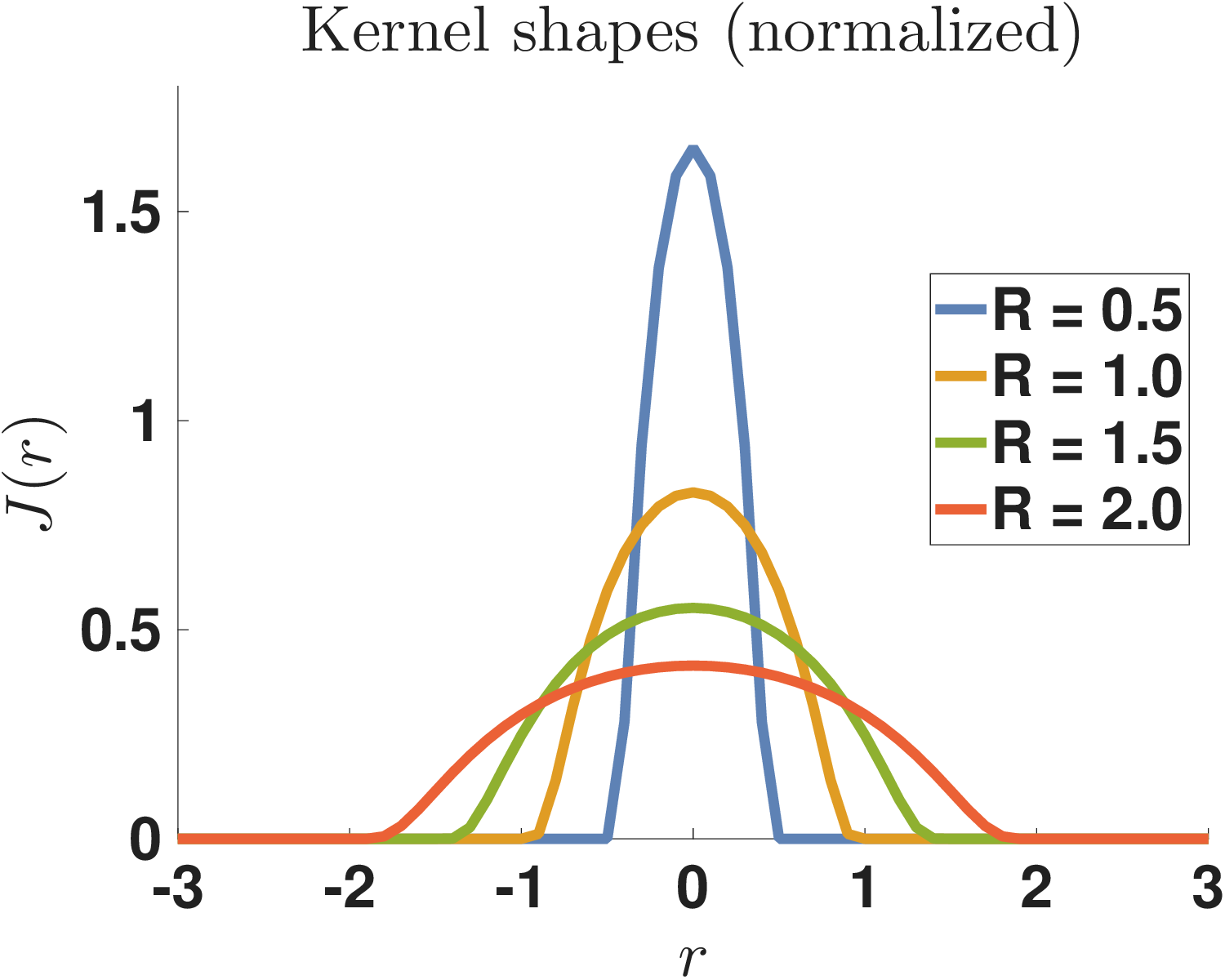}\includegraphics[width=0.3333\linewidth]{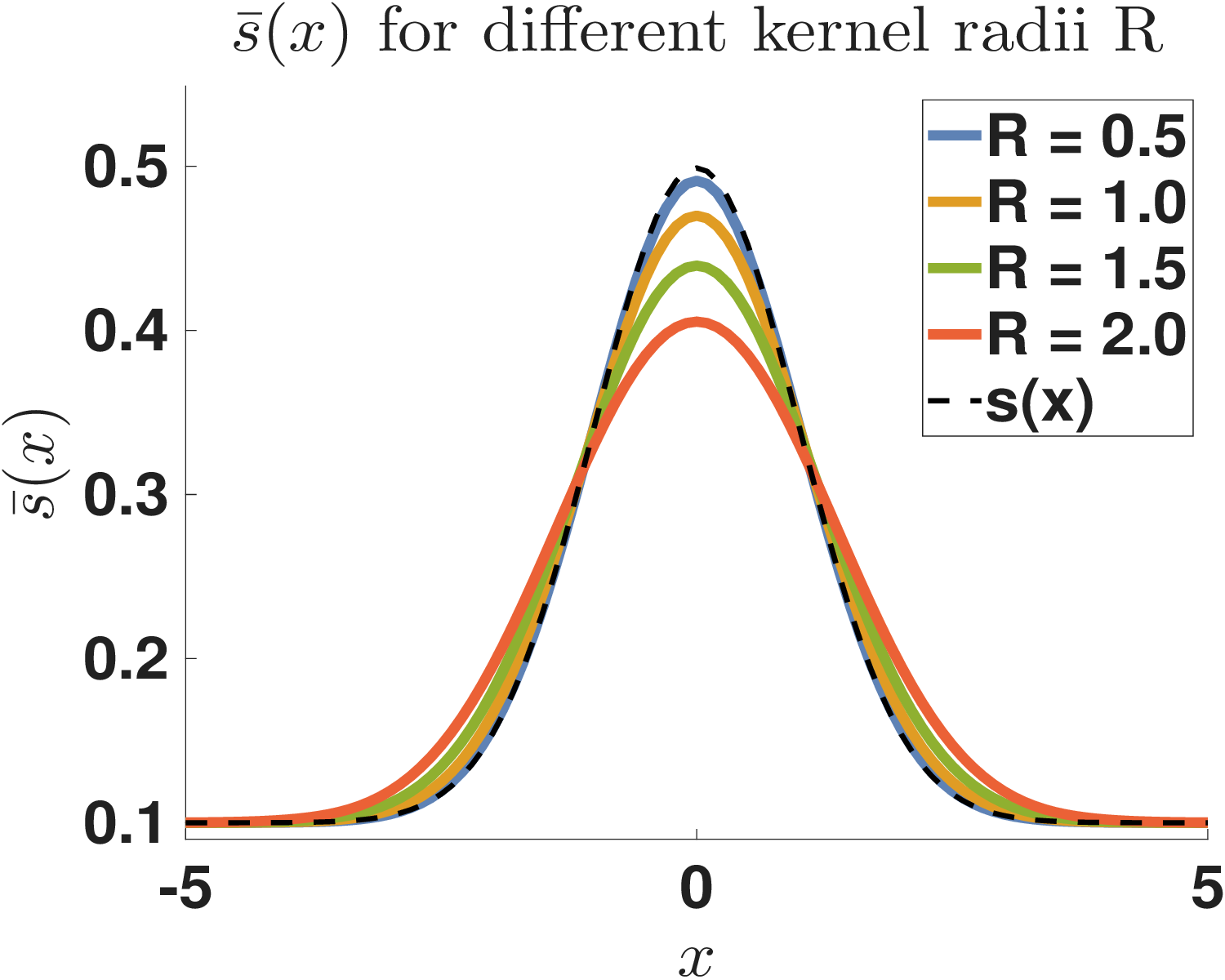}~\includegraphics[width=0.3333\linewidth]{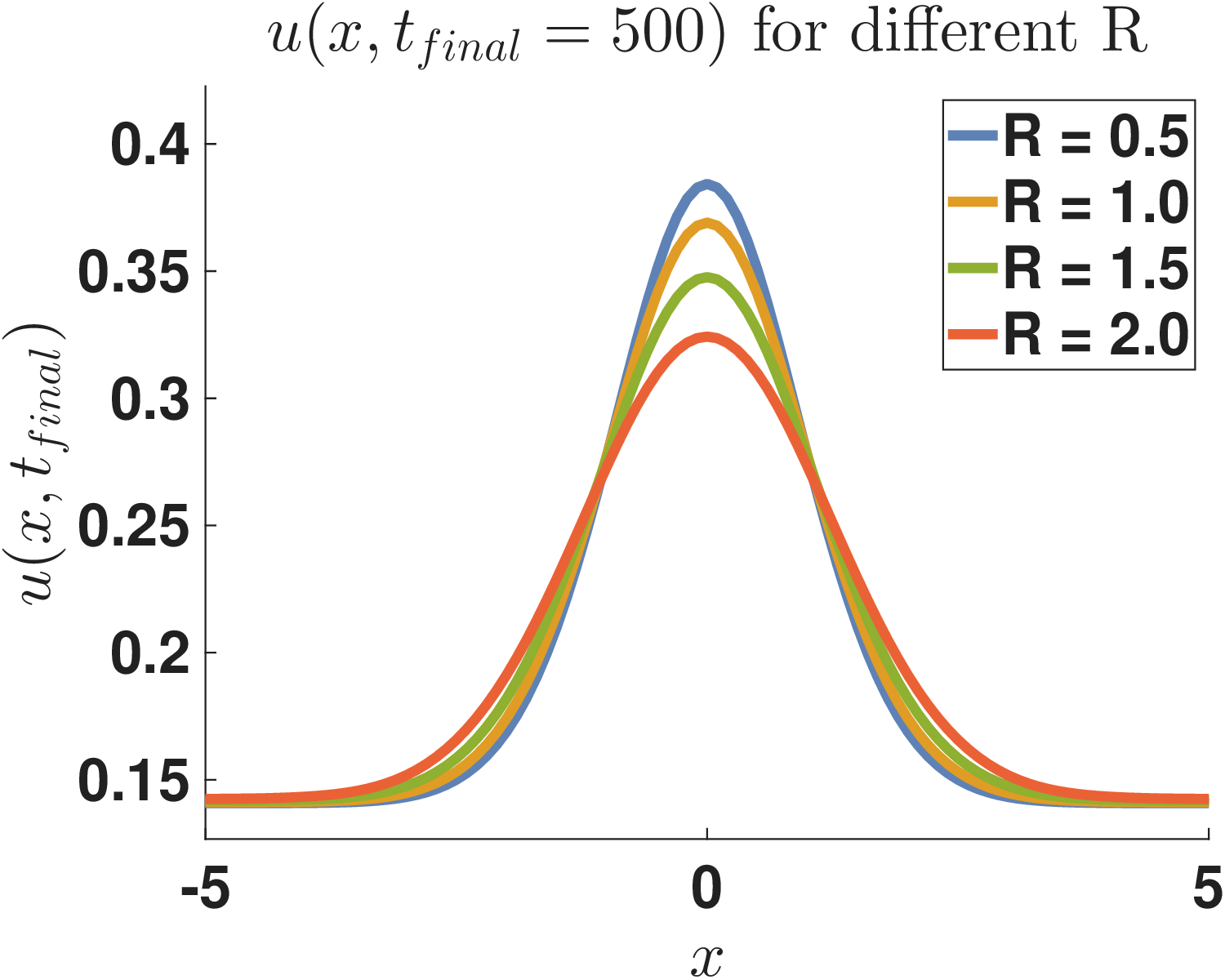}
    \caption{The left panel shows normalized kernels with $R=0.5,1,1.5,$ and $2$. In the second figure, the environment is given as $s(x) = \frac{1}{\sqrt{2\pi}}e^{-x^2/2} + 0.1$. $\bar s(x)$ closely tracks $s(x)$, with mild variation across $R$. In the last figure, population distributions are similar across $R$. The main shape of the profiles comes from $s(x)$, not the boundary truncation.}
    \label{fig:hetero s}
\end{figure}

The non-normalized kernel aggregates the total amount of information within the perceptual range. Near the boundary, the accessible region is smaller, so individuals obtain less information even if the underlying $s(x)$ is constant. Thus, boundary locations are perceived as less favorable. On the other hand, the normalized kernel corresponds to sensing the average concentration in the neighborhood. Even if the perceptual ball is truncated near the boundary, the average equals $s$ in a homogeneous environment, so boundary and interior are indistinguishable.

Individuals move according to their cognitive map, and this shapes the resulting population distribution. In this regard, the structure of the cognitive map becomes more important than the geometric shape of the environment itself. In the next subsection, our focus is therefore on lingering behaviour over a heterogeneous cognitive map rather than on boundary effects, and we accordingly adopt the normalized kernel for numerical simulations.

\subsection{The lingering phenomenon on heterogeneous landscapes}
In this subsection, we study how the learning and forgetting rates $\alpha$ and $\mu$ shape the lingering phenomenon on a given heterogeneous environment. We express the heterogeneous environment through the cognitive map $\bar s(x)$ as defined in \eqref{bars}. The spatial structure of the cognitive map \(\bar s(x)\) is fixed, and our main interest is how the parameters \(\alpha\) (learning) and \(\mu\) (forgetting) modulate the long-time distribution of individuals on this landscape. We fix the normalized perception kernel and consider a one-dimensional symmetric domain $\Omega=(-\ell,\ell)$. Our goal is to characterize the lingering phenomenon, where the steady-state density $u$ becomes more concentrated near favourable locations of $\bar s$ for intermediate values of the memory parameter $\mu$, while very small or very large $\mu$ lead to more uniform distributions.

We therefore study the steady-state problem of \eqref{eqn: dispersalonly}:
\begin{equation}\label{eqn: SS dispersal}
\left\{
\begin{aligned}
&0  = \big(\gamma(m_\infty) u_\infty\big)''
&& \text{in } (-\ell,\ell)\times(0,\infty), \\[4pt]
&0 = \big(\alpha\,\bar s(x)-m_\infty\big)\,u_\infty - \mu\,m_\infty
&& \text{in } (-\ell,\ell)\times(0,\infty),\\
&\big(\gamma(m_\infty) u_\infty\big)'(\pm\ell) = 0
&& \text{on } (0,\infty)
\end{aligned}
\right.
\end{equation}
and assume that the solutions to \eqref{eqn: dispersalonly} converge to a solution to a steady-state solving \eqref{eqn: SS dispersal} as $t\to \infty$. By \cite[Theorem 7.3]{amann1990}, the solution $u$ of \eqref{eqn: dispersalonly} is $C^1$ in $\alpha,\mu$, hence we also assume that $u_\infty$ is $C^1$ in terms of $\alpha,\mu$. We view the steady states as functions of $x$ parameterized by $(\alpha,\mu)$, i.e.,
$u_\infty=u_\infty(x;\alpha,\mu)$ and $m_\infty=m_\infty(x;\alpha,\mu)$. Since the homogeneous Neumann boundary condition implies conservation of total mass for \eqref{eqn: dispersalonly}, the limiting steady state satisfies that \begin{equation}
    \label{mass conservation}
    \int_{-\ell}^\ell u_\infty(x;\alpha,\mu) dx = M,
\end{equation} where $M>0$ is determined by the initial mass of \eqref{eqn: dispersalonly} and hence is independent of $\alpha$ and $\mu$.

Before we present the main lingering result of this section, we consider an illustrative numerical example. 
For simulations, we choose $\ell=5$ and the resource distribution $s(x)$ is chosen as
\begin{equation*}
    s(x) = \frac{1}{\sqrt{2\pi}}e^{-x^2/2}+0.1 \,\,\text{ for }x\in[-5,5].
\end{equation*} We choose $\gamma$ as in \eqref{gamma} and the normalized kernel $K(x,y)=J_R(x-y)/Z_R(x)$, where $J_R$ is  \eqref{bump} with the perceptual radius $R=1.5$.

Figure \ref{fig:lingering} illustrates the effect of the forgetting rate $\mu$ for two values of the learning rate $\alpha=1,10$.  
 We choose the initial condition 
$u_0(x)=s(x)$ and compute the solution of \eqref{eqn: dispersalonly}
up to $t_{\text{final}}=500$. On the left ($\alpha=1$), the stationary densities $u(x,t_{\text{final}})$ remain close to a flat profile for all $\mu$. Memory is too weak to produce a pronounced peak. On the right ($\alpha=10)$, however, intermediate values of $\mu$ lead to a strong accumulation near the maximum of $\bar s(x)$, whereas very small and very large 
$\mu$ give flatter profiles. 

To quantify this behavior, Figure \ref{fig:maxu_mu} plots the maximum of $u(x,t_{\text{final}})$ as a function of $\mu$. For $\alpha=1$ (left), the curve decreases monotonically with increasing forgetting rate $\mu$ and lingering does not occur. For $\alpha=10$ (right), the curve first increases and then decreases, showing that an intermediate memory time maximizes the peak density at the resource. This non-monotone dependence of $u_\infty(0;\alpha,\mu)$ is what we refer to as the \emph{lingering phenomenon}.
\begin{figure}[ht]
    \centering
    \includegraphics[width=0.4\linewidth]{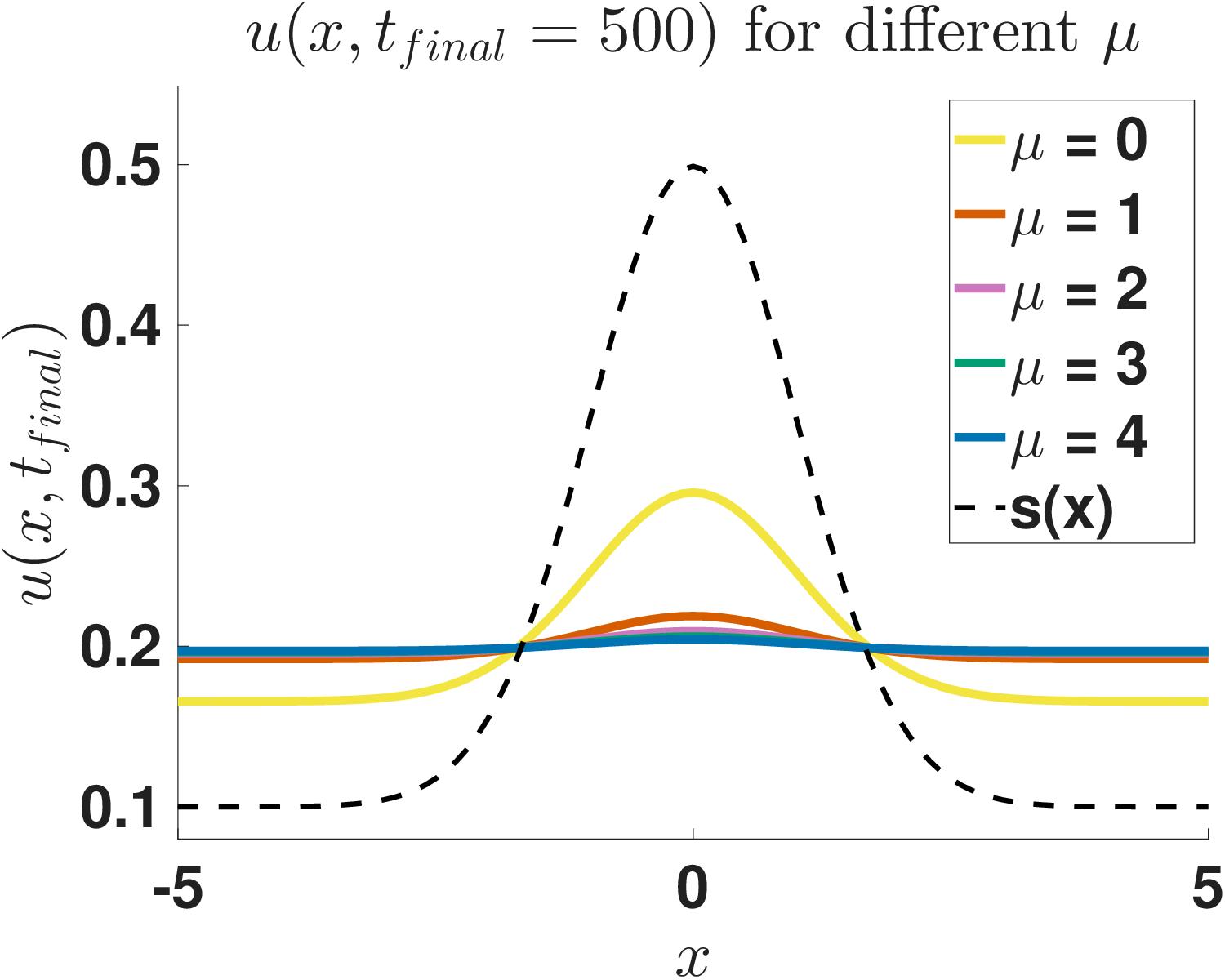}
    \includegraphics[width=0.4\linewidth]{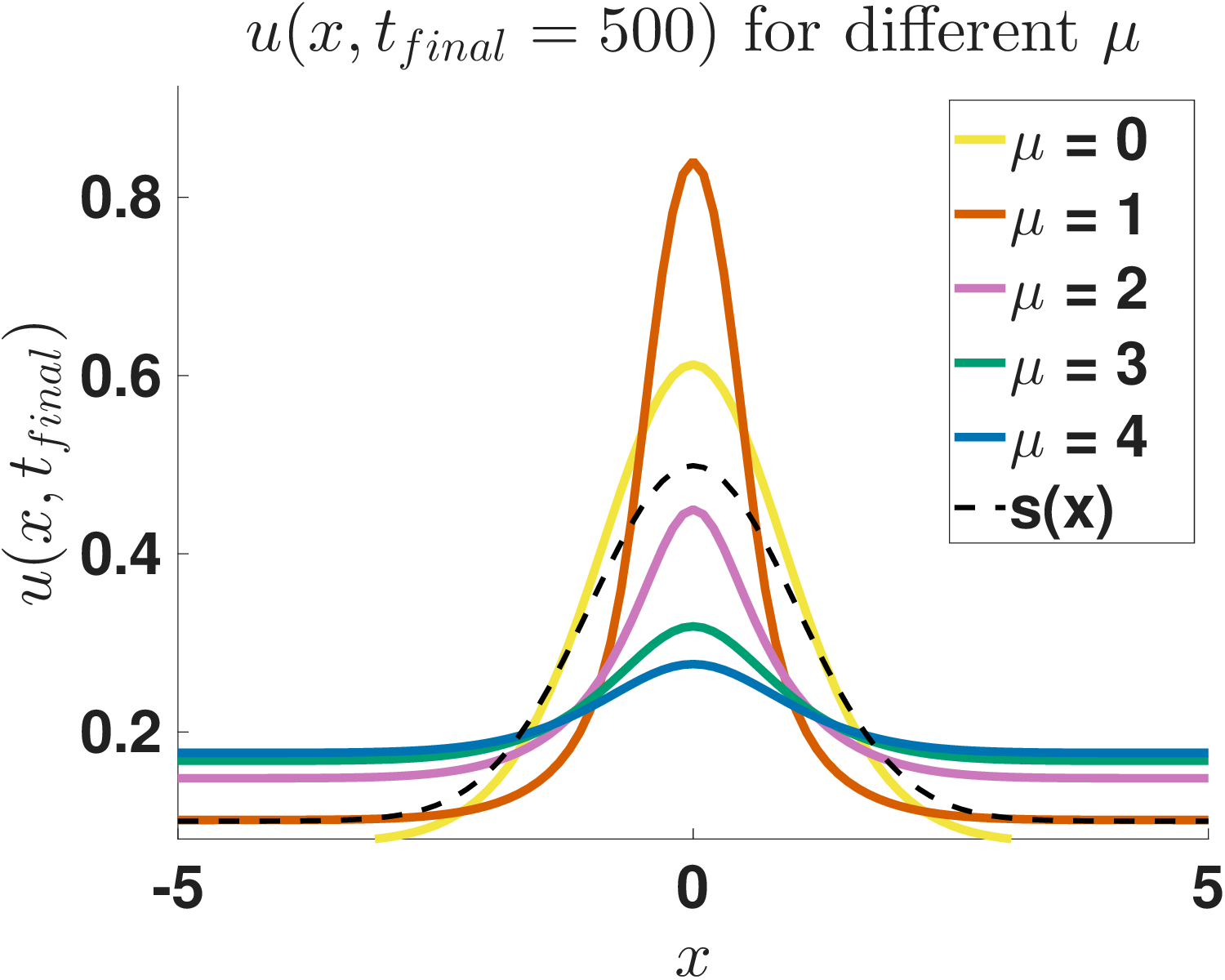}
    
    \caption{Snapshots of $u(x,t_{\mathrm{final}}=500)$ for different values of the forgetting rate $\mu$. The left panel corresponds to $\alpha=1$ and the right panel to $\alpha=10$. The initial condition is $u_0=s$ and we use the bump kernel \eqref{bump} with radius $R=1.5$.
}
    \label{fig:lingering}
\end{figure}

\begin{figure}[ht]
    \centering
    \includegraphics[width=0.4\linewidth]{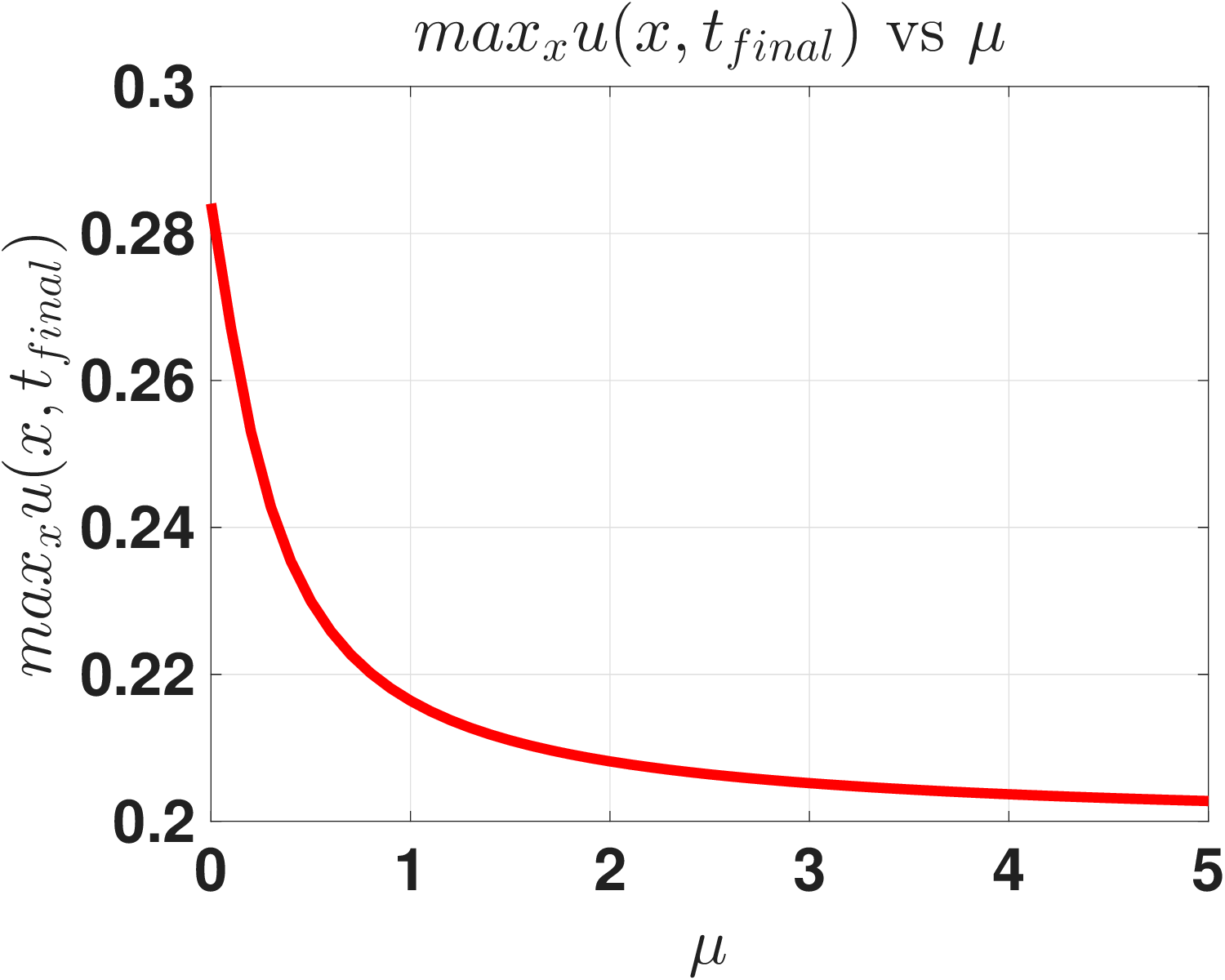}~\includegraphics[width=0.4\linewidth]{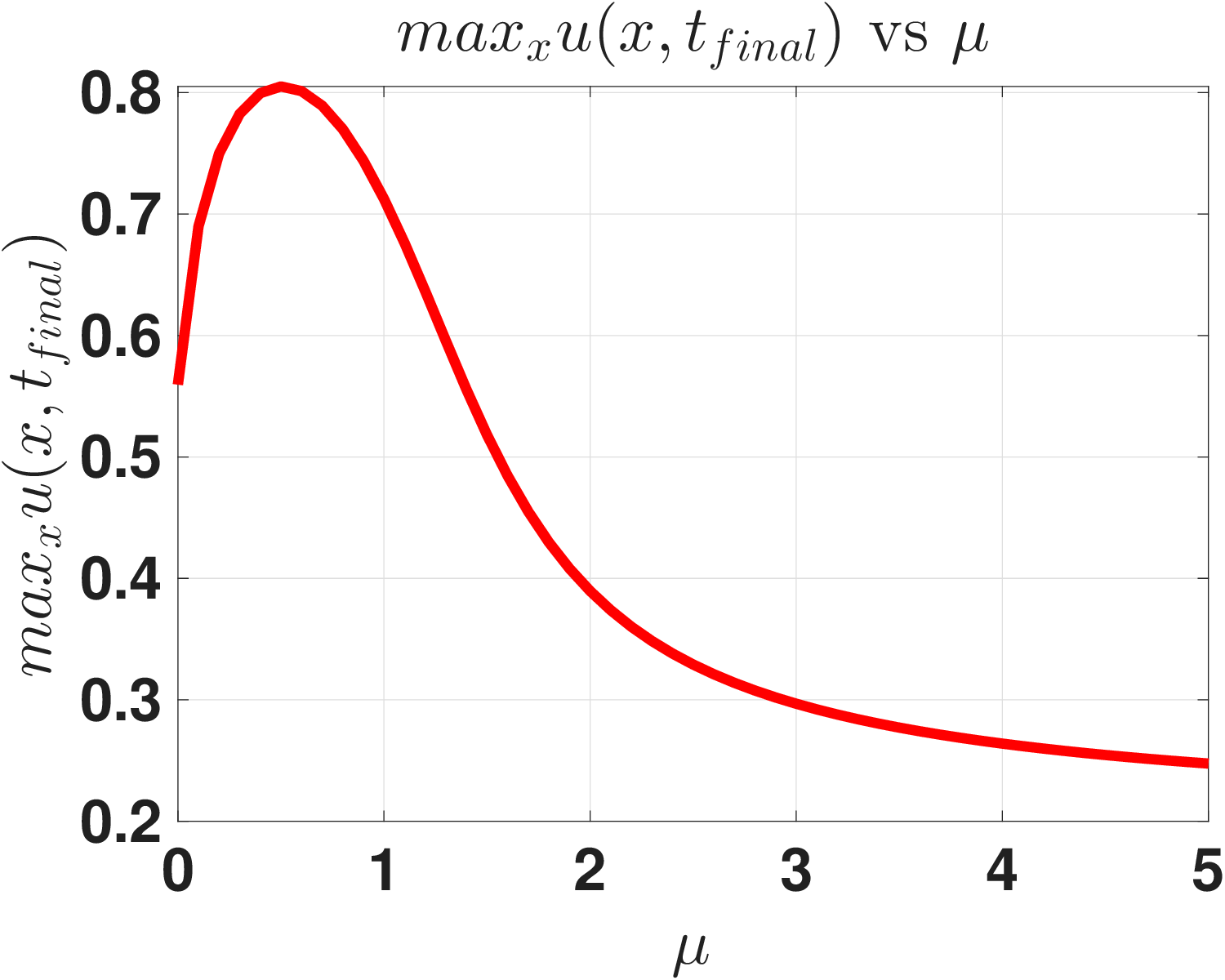}
    \caption{Maximum population density $\max_{x\in\Omega}u(x,t_{\mathrm{final}})$ as a function of the forgetting rate $\mu$. The left panel corresponds to $\alpha=1$ and the right panel to $\alpha=10$. The initial condition is $u_0=s$ and we sample $\mu=0,0.1,0.2,\dots,4.9,5$.
}
    \label{fig:maxu_mu} 
\end{figure}

\begin{theorem}[Lingering phenomenon on unimodal $\bar s(x)$]\label{thm:lingering}
Let $\gamma\in C^2[0,\infty)$. Assume that $\bar s(x)$ is even in $(-\ell,\ell)$ and decreasing in $(0,\ell)$. Given $\alpha>0,\mu\ge0$, let $(u_\infty(\cdot;\alpha,\mu),m_\infty(\cdot;\alpha,\mu))$ be the solution of \eqref{eqn: SS dispersal} that satisfies \eqref{mass conservation}. 
\begin{enumerate}
\item[(i)] There exist constants $\underbar{u}(\alpha),\bar u(\alpha) >0$ such that $ \underbar u(\alpha)\le u_\infty(x;\alpha,\mu)\le \bar u(\alpha)$ for all $x\in \Omega$ and $\mu\ge 0$.
\item[(ii)]  $0<m_\infty(x;\alpha,\mu)\le \alpha\bar s(x)$ for all $\mu\ge 0$. Furthermore, $m_\infty(\cdot;\alpha,\mu)\to 0$ uniformly as $\mu\to\infty$.

    \item [(iii)] There exist constants $C>0$ and $\mu_0$ such that for all $\mu\ge \mu_0$, $u_\infty$ has its maximum at $x=0$, and
\[
    \frac{\partial u_\infty}{\partial \mu}(0;\alpha,\mu)<0 \text{ and } \bigl|\frac{\partial u_\infty}{\partial \mu}(x;\alpha,\mu)\bigr| \;\le\; \frac{C}{\mu^2}
    \qquad \text{for all } x\in[-\ell,\ell].
\]
In particular, $u_\infty(\cdot;\alpha,\mu)\to \frac{M}{2\ell}$ uniformly as $\mu\to\infty$. 
    
    \item[(iv)] If $z\mapsto -z\gamma'(z)$ is increasing on $(0,z_0]$ and decreasing on $(z_0,\infty)$, then there exists $\alpha_1<\alpha_2$ such that $\frac{\partial u_\infty}{\partial \mu}(0;\alpha,0)<0$ for $\alpha <\alpha_1$ and $\frac{\partial u_\infty}{\partial \mu}(0;\alpha,0)>0$ for $\alpha >\alpha_2$.
\end{enumerate}
\end{theorem}

\begin{proof}
\textbf{(i)} 
Recall from \eqref{mass conservation} that the constant $M=\int_\Omega u_\infty(x;\alpha,\mu)dx$ is independent of $\alpha$ and $\mu$. Let $x\in(-\ell,\ell)$ and $\alpha,\mu>0$ be fixed. 
By using the boundary condition again, there exists $K=K(\alpha,\mu)>0$ independent of $x$ such that \begin{equation}\label{eqn:SSrel}
     \gamma(m_\infty)u_\infty=K.
 \end{equation} 
   By using \eqref{eqn:SSrel}, the constant $K$ can be expressed by
\[K = \frac{M}{\int_\Omega\frac{1}{\gamma(m_\infty(x;\alpha,\mu))}dx}.\] Since $\gamma$ is bounded below by $\gamma(\alpha\max  \bar s)$ and bounded above by $\gamma(0)$, the map $\mu\mapsto K(\alpha,\mu)$ is bounded, that is, \[ \frac{\gamma(\max{\alpha \bar s}) M}{|\Omega|}\le K(\alpha,\mu)\le \frac{\gamma(0) M}{|\Omega|}.\] The bounds for $K$, $\gamma$ and \eqref{eqn:SSrel} ensures that
\[\frac{\gamma(\max{\alpha \bar s}) M}{\gamma(0)|\Omega|}\le u_\infty(x;\alpha,\mu)\le \frac{\gamma(0) M}{\gamma(\alpha \max{\bar s})|\Omega|}. \]

\textbf{(ii)} The equality $m_\infty = \alpha\bar s\dfrac{u_\infty}{u_\infty +\mu}$ directly yields the bound. Since $u_\infty$ is uniformly bounded in $\mu$, $m_\infty(\cdot;\alpha,\mu) \to 0$ uniformly in $\overline{\Omega}$ as $\mu\to \infty$.

\textbf{(iii)} Recall the constant $K=K(\alpha,\mu)$ from \eqref{eqn:SSrel}. Let $F(u,K,\alpha,\mu) = \gamma(m)u - K$, so we have $F(u_\infty,K,\alpha,\mu)=0$ in $\overline{\Omega}$. Differentiate both sides with respect to $\mu$:
\begin{align}\label{eqn: total deriv}
    \frac{\partial F}{\partial u}  \frac{\partial u_\infty}{\partial \mu} +\frac{\partial F}{\partial K} \frac{\partial K}{\partial \mu}  +\frac{\partial F}{\partial \mu}  =0,
\end{align} where
\begin{align*}
    \frac{\partial F}{\partial u}  = \gamma(m_\infty) +\alpha \bar s\gamma'(m_\infty) \frac{\mu u_\infty}{(u_\infty+\mu)^2}, \,\,\frac{\partial F}{\partial K}  = -1, \text{ and } \frac{\partial F}{\partial \mu}  = -\alpha\bar s\gamma'(m_\infty) \frac{u_\infty^2}{(u_\infty+\mu)^2}.
\end{align*}
Note that $m_\infty$ and $u_\infty$ are uniformly bounded on $\overline{\Omega}$, and $m_\infty \to 0$ as $\mu\to\infty$ by (i) and (ii). Since $\gamma\in C^2[0,\infty)$, $\gamma(m_\infty)\to\gamma(0)$ and $\frac{\partial F}{\partial u} (x)\to \gamma(0)$ uniformly as $\mu\to\infty$.  Choose a sufficiently large $\mu_0>0$ such that for some $c_0>0$,
\[\frac{\partial F}{\partial u}  = \gamma(m_\infty) +\alpha \bar s\gamma'(m_\infty) \frac{\mu u_\infty}{(u_\infty+\mu)^2}>c_0,\quad \frac{\partial F}{\partial u} =\gamma(0)+O(\mu^{-1}), \text{ and, }\]\[\frac{\partial F}{\partial \mu}  = -\alpha\bar s\gamma'(m_\infty)\frac{u_\infty^2}{(u_\infty+\mu)^2}\le \frac{C}{\mu^2} \] for all  $ \mu>\mu_0$, where $C$ is a positive constant independent of $\mu$. We  differentiate both sides of $\gamma(m_\infty)u_\infty=K$ with respect to $x$ to have
\[ \frac{\partial F}{\partial u}u' = -\alpha\frac{u_\infty}{u_\infty+\mu}\gamma'(m_\infty)\bar s'. \] Since $\frac{\partial F}{\partial u}>0$ and $ \gamma'(m_\infty)<0$, $u'$ and $\bar s'$ have the same sign. Therefore, $u$ has its maximum at $x=0$.

The mass constraint from (i) implies
\[
0 = \frac{d}{d\mu}\int_{-\ell}^{\ell} u_\infty\,dx
= \int_{-\ell}^{\ell} \frac{\partial u_\infty}{\partial \mu}\,dx
= \int_{-\ell}^{\ell} \frac{\frac{\partial K}{\partial \mu}  - \frac{\partial F}{\partial \mu} }{\frac{\partial F}{\partial u} }\,dx.
\]
Solving for $\frac{\partial K}{\partial \mu} $ gives
\[
\frac{\partial K}{\partial \mu} 
=
\frac{\displaystyle\int_{-\ell}^{\ell} \frac{\frac{\partial F}{\partial \mu} }{\frac{\partial F}{\partial u} }\,dx}
{\displaystyle\int_{-\ell}^{\ell} \frac{1}{\frac{\partial F}{\partial u} }\,dx}.
\]
The mass constraint $M=\int_{-\ell}^{\ell}u_\infty dx$ and \eqref{eqn:SSrel} gives
\[u_\infty = \frac{K(\mu)}{\gamma(m_\infty)} = \displaystyle\frac{M}{\gamma(m_\infty)\int_{-\ell}^{\ell}\frac{1}{\gamma(m_\infty)}dx}\to \frac{M}{\gamma(0)\frac{2\ell}{\gamma(0)}}=\frac{M}{2\ell}\] uniformly in $x$.

We now claim that $\frac{\partial u_\infty}{\partial \mu}<0$ for sufficiently large $\mu$. For large $\mu$, 
\[\frac{\partial F}{\partial \mu}  = -\alpha\bar s\gamma'(m_\infty) \frac{u_\infty^2}{(u_\infty+\mu)^2}=-\frac{\alpha\bar s(x)\gamma'(0)M}{2\ell\mu^2}+O(\mu^{-2})\] uniformly in $\overline{\Omega}$. Therefore, for large $\mu$
\[\frac{\partial K}{\partial \mu}  = \frac{\displaystyle\int_{-\ell}^{\ell}\frac{\frac{\partial F}{\partial \mu} }{\frac{\partial F}{\partial u} }dx}{\displaystyle\int_{-\ell}^{\ell}\frac{1}{\frac{\partial F}{\partial u} }dx}=\frac{\displaystyle\int_{-\ell}^{\ell}\frac{\frac{\partial F}{\partial \mu} }{\gamma(0)}dx}{\displaystyle\int_{-\ell}^{\ell}\frac{1}{\gamma(0)}dx} +O(\mu^{-1})= \overline{\frac{\partial F}{\partial \mu} } + O(\mu^{-1}),\] where $\overline{\frac{\partial F}{\partial \mu} } = \frac{1}{2\ell}\int_{-\ell}^\ell \frac{\partial F}{\partial \mu} (x)dx$. Solve \eqref{eqn: total deriv} for $\frac{\partial u_\infty}{\partial \mu}$, and we obtain
\[\frac{\partial u_\infty}{\partial \mu} = \frac{\frac{\partial K}{\partial \mu}  - \frac{\partial F}{\partial \mu} (x)}{\frac{\partial F}{\partial u} } = \frac{\overline{\frac{\partial F}{\partial \mu} }-\frac{\partial F}{\partial \mu} (x)}{\gamma(0)}+O(\mu^{-2}).  \]
Since $u_\infty$ and $\bar s$ has its maximum at $x=0$, $\overline{\frac{\partial F}{\partial \mu} }<\frac{\partial F}{\partial \mu} (x)$ for large $\mu$, hence $\frac{\partial u_\infty}{\partial \mu}<0$ for sufficiently large $\mu$. 

Using $\frac{\partial F}{\partial u} >0$ and $|\frac{\partial F}{\partial \mu} |\le C/\mu^2$, we obtain
\[
|\frac{\partial K}{\partial \mu} |
\le
\frac{\displaystyle\int_{-\ell}^{\ell} \frac{|\frac{\partial F}{\partial \mu} |}{\frac{\partial F}{\partial u} }\,dx}
{\displaystyle\int_{-\ell}^{\ell} \frac{1}{\frac{\partial F}{\partial u} }\,dx}
\le
\frac{\displaystyle\frac{C}{\mu^2}\int_{-\ell}^{\ell} \frac{1}{\frac{\partial F}{\partial u} }\,dx}
{\displaystyle\int_{-\ell}^{\ell} \frac{1}{\frac{\partial F}{\partial u} }\,dx}
= \frac{C}{\mu^2}.
\]

Finally, substituting it back into the expression for $\partial_\mu u$,
\[
|\partial_\mu u|
\le
\frac{|\frac{\partial K}{\partial \mu} | + |\frac{\partial F}{\partial \mu} (x,\mu)|}{\frac{\partial F}{\partial u} (x,\mu)}
\le
\frac{1}{c_0}\left(\frac{C}{\mu^2} + \frac{C}{\mu^2}\right)
= \frac{2C}{\mu^2},
\]
for some constant $C>0$ independent of $\mu$ (for all $\mu\ge \mu_0$). This proves the claimed $O(\mu^{-2})$ decay and uniform convergence $\partial_\mu u(\cdot,\mu)\to 0$ as $\mu\to\infty$.

\textbf{(iv)}  At $\mu=0$, we have 
\begin{equation*}
    \frac{\partial F}{\partial u}  = \gamma(\alpha \bar s)>0, \text{ and } \frac{\partial F}{\partial \mu}  = -\alpha\bar s\gamma'(\alpha\bar s)>0.
\end{equation*} Therefore, $u$ has its maximum at $x=0$ by the relation $\gamma(\alpha\bar s)u' = -\alpha\gamma'(\alpha \bar s)\bar s'$.  
If we combine the constraint and \eqref{eqn: total deriv}, we obtain 
\begin{align*}
    \frac{\partial K}{\partial \mu} \Big|_{{\mu=0}} = \frac{\int_{-\ell}^{\ell}\frac{\partial F}{\partial \mu}  \frac{\partial F}{\partial u} ^{-1}dx}{\int_{-\ell}^{\ell}\frac{\partial F}{\partial u} ^{-1}dx} =\frac{\int_{-\ell}^{\ell}-\alpha\bar s\gamma'(\alpha\bar s) \gamma(\alpha \bar s)^{-1}dx}{\int_{-\ell}^{\ell}\gamma(\alpha \bar s)^{-1}dx}.
\end{align*} Notice that $\frac{\partial K}{\partial \mu} $ is a weighted average of $-\alpha\bar s \gamma'(\alpha \bar s)$ over $(-\ell,\ell)$. We use \eqref{eqn: total deriv} again to have
\begin{equation*}
    \partial_\mu u\Big|_{\mu=0} = \frac{\frac{\partial K}{\partial \mu} \big|_{\mu=0} -(-\alpha\bar s\gamma'(\alpha\bar s))}{\gamma(\alpha\bar s)},
\end{equation*} which implies that the sign of $\partial_\mu u\big|_{\mu=0}$ is determined by the sign of $\frac{\partial K}{\partial \mu} \big|_{\mu=0} -(-\alpha\bar s\gamma'(\alpha\bar s))$. Suppose that the map $z\mapsto -z\gamma'(z)$ is increasing on $(0,z_0)$ and decreasing on $(z_0,\infty)$. Choose $\alpha_1 = z_0/\bar s(0)$ so that  $(\alpha\bar s(\ell),\alpha\bar  s(0))\subset (0,z_0)$ for $\alpha<\alpha^*$. Since $\bar s(0)$ is the maximum of $\bar s$, for $\alpha<\alpha_1$
\begin{equation*}
    -\alpha\bar s(x)\gamma'(\alpha\bar s(x)) - (-\alpha\bar s(0)\gamma'(\alpha\bar s(0))),
\end{equation*} hence $\frac{\partial K}{\partial \mu} \big|_{\mu=0} -(-\alpha\bar s\gamma'(\alpha\bar s))<0$, so $\partial_\mu u\big|_{\mu=0}<0$ at $x=0$. Let $\alpha_2 = z_0/\bar s(\ell)$. Then $(\alpha \bar s (\ell),\alpha \bar s (0))\subset (z_0,\infty)$. Similarly, for $\alpha>\alpha_2$, we have $\partial_\mu u\big|_{\mu=0}>0$ at $x=0$.

\end{proof} 
\begin{remark}
Parts (iii) and (iv) formalize the lingering phenomenon observed in Figure \ref{fig:maxu_mu}. For small 
$\alpha$, the peak density 
$u_\infty(0;\alpha,\mu)$ decreases with $\mu$, whereas for large $\alpha$ it first increases and then decreases, so an intermediate value of $\mu$ maximizes the peak.
\end{remark}

\begin{remark}
    Let $c>0,k>0$ and $\gamma(z) = \frac{1}{(z+c)^k}$. Then the map $z\mapsto -z\gamma'(z) = \frac{kz}{(z+k)^{k+1}}$ is increasing on $(0,\frac{c}{k})$ and decreasing on $(\frac{c}{k},\infty)$.
\end{remark}

This type of lingering can also arise  in asymmetric landscapes. Figure \ref{fig:asymmetric} shows an asymmetric memorized resource distribution $\bar s(x)$. For moderate memory ($\mu=0.5)$, individuals are more concentrated near the dominant peak, and this bias is stronger for a larger learning rate $\alpha$. The lingering effect is consistent with Theorem \ref{thm:lingering} at the dominant peak, but here the asymmetry of 
$\bar s(x)$ introduces an additional directional bias. 
\begin{figure}[ht]
    \centering
   \includegraphics[width=0.4\linewidth]{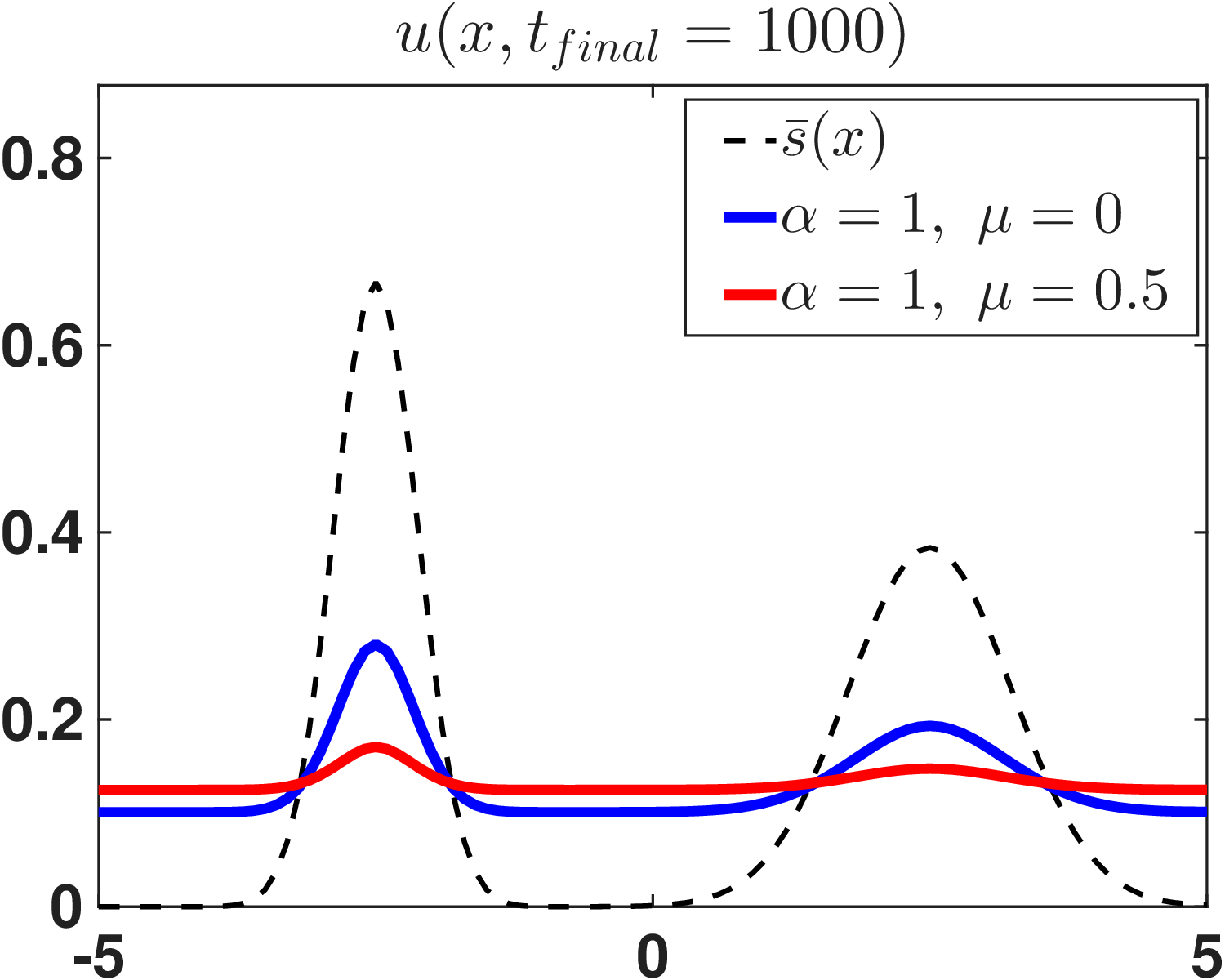}~
    \includegraphics[width=0.4\linewidth]{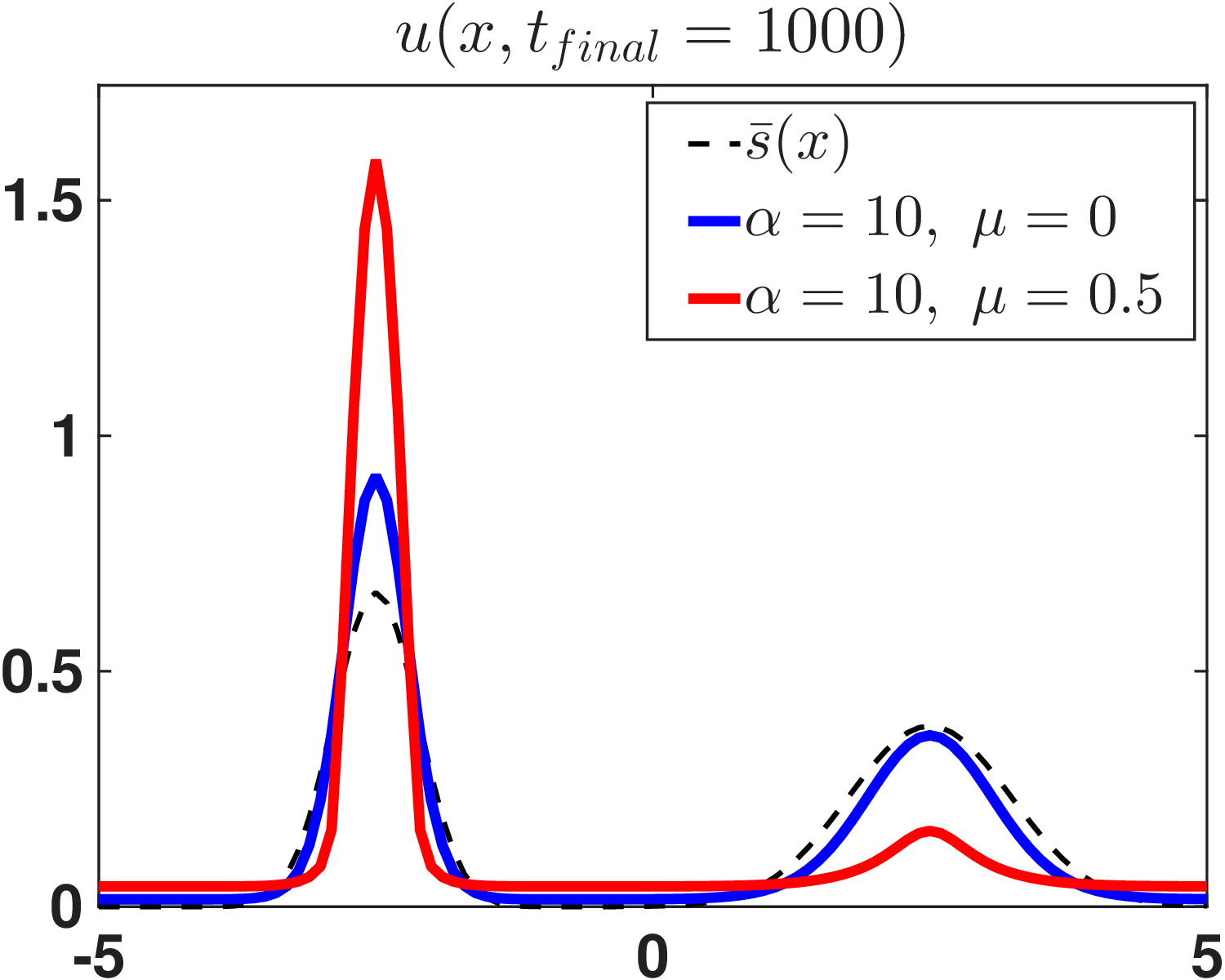}
    \caption{Numerical solutions $u(x,t_{\mathrm{final}}=1000)$ for an asymmetric memorized resource distribution $\bar s(x)$ (dashed curves). 
The normalized perception kernel with radius $R=1$ is used in all simulations. 
Left panel: learning rate $\alpha=1$ with forgetting rates $\mu=0$ and $\mu=0.5$. 
Right panel: learning rate $\alpha=10$ with $\mu=0$ and $\mu=0.5$. 
For $\alpha=10$ the case $\mu=0.5$ produces a sharper peak near the dominant resource patch, indicating stronger lingering on the heterogeneous cognitive map.
}
    \label{fig:asymmetric}
\end{figure}

\section{Population dynamics model}\label{sec:population dynamics}
In the previous section, we studied lingering behaviour in a purely movement-based model without births and deaths. We now incorporate logistic population dynamics into the Fokker--Planck type diffusion model to examine how demographic growth interacts with memory-driven movement.

We consider the reaction--diffusion model:
\begin{equation}\label{eqn:logistic}
\left\{\begin{aligned}
    &u_t = \Delta (\gamma(m)u) + f(s,u),  && \text{ in }\Omega \times (0,T),\\
    &m_t = (\alpha \bar s -  m) u-\mu m,\\
    &\nabla (\gamma(m)u) \cdot \textbf{n}=0, && \text{ on }\partial \Omega\times [0,T],\\
    &u(x,0) = u_0(x)>0, m(x,0) = m_0(x)\ge 0, &&\text{ in } \Omega,
\end{aligned}\right.
\end{equation} where $s=s(x)$ is a given food distribution. We also assume that $s(x)>0$ for all $x\in \Omega$ and $m_0(x) \le \alpha \bar s(x)$. The kinetic term $f(s,u)$ is of the form:
\begin{equation*}
    f(s,u) = g(s,u)u,
\end{equation*} where $g(s,0)= s$ and the map $u\mapsto g(s,u)$ is decreasing, Lipchitz continuous in $u$, and 
\[ |g(s,u)|\le C(1+|u|). \]
The function $g(s,u)$ represents the per-capita growth rate in an environment with local resource level $s(x)$. The condition $g(s,0)=s$ means that the initial growth rate is proportional to resource availability, while the monotonicity of the map $u\mapsto g(s,u)$ encodes crowding effects, meaning that higher density reduces per-capita growth. Lipschitz continuity and the growth bound are technical assumptions for well-posedness.

The above model (\ref{eqn:logistic}) has a few special cases that were studied before: \\

\noindent{\bf Example 1:} 
    When $\gamma(z)=\gamma$ for some positive constant $\gamma>0$ and $g(s,u)=s-u$ the system \eqref{eqn:logistic} reduces to a single reaction--diffusion equation
    \begin{equation}\label{eqn:constantgamma}
\left\{\begin{aligned}
    &u_t = \gamma\Delta u + u(s-u),  && \text{ in }\Omega \times (0,T),\\
    &\nabla u \cdot \textbf{n}=0, && \text{ on }\partial \Omega\times [0,T],\\
    &u(x,0) = u_0(x)>0, &&\text{ in } \Omega.
\end{aligned}\right.
\end{equation} The equation \eqref{eqn:constantgamma} is a classical Fisher-KPP equation (or diffusive--logistic equation) \cite{Murray2002}. It is well known that if $s(x)\ge 0$, then there exists a positive global equilibrium $u^*(x)$ for a proper diffusion rate $\gamma$ \cite{Cantrell1989,Lou2006}.

\noindent{\bf Example 2:} 
    When $m=s(x)$ and $g(s,u)=s-u$ the system \eqref{eqn:logistic} reduces to a single reaction--diffusion equation \begin{equation}\label{eqn:resource dependent}
\left\{\begin{aligned}
    &u_t = \Delta(\gamma(s(x)) u) + u(s(x)-u),  && \text{ in }\Omega \times (0,T),\\
    &\nabla (\gamma(s)u) \cdot \textbf{n}=0, && \text{ on }\partial \Omega\times [0,T],\\
    &u(x,0) = u_0(x)>0, &&\text{ in } \Omega.
\end{aligned}\right.
\end{equation} 

In \cite{Tang2023} the global dynamics of \eqref{eqn:resource dependent} and the total-mass property of its steady-state solution was studied. If we choose $\alpha=1$ and $\mu=0$ and assume that $\bar s = s$, then $m(x,t)\equiv s(x)$ is a steady cognitive map and \eqref{eqn:logistic} coincides with \eqref{eqn:resource dependent}. Biologically, this corresponds to individuals that use only local habitat information and do not forget.

\subsection{Local existence}
We first verify the well posedness of model \eqref{eqn:logistic} mathematically by proving the local existence of its unique solution.
To guarantee uniform ellipticity of the diffusion operator in the equation, we have the following hypothesis:
\begin{equation}
\label{hyp: gamma}    
\text{The  map } m\mapsto\gamma(m)>0 \text{ is $C^3([0,\infty))$, decreasing,}
\end{equation}
\begin{equation}
\label{hyp: id}    
\begin{aligned}
u_0\in C^{2+\theta}(\overline{\Omega}),\quad m_0\in C^{2+\theta}(\overline{\Omega}) \text{ for some } \theta\in(0,1) \text{ and } m_0\le \alpha \bar s.
\end{aligned}
\end{equation}
\begin{theorem}Let $0<\theta<1$. Let $\Omega\subset \mathbb{R}^n$ be a bounded domain with $C^{2+\theta}$ boundary. Assume  \eqref{hyp: gamma} and $s\in C^{2+\theta}(\overline{\Omega})$. Then there exists a unique nonnegative solution $(u,m)\in [C^{2+\theta,1+\theta/2}(\overline{\Omega}\times[0,T_{\max}))]^2$ to \eqref{eqn:logistic}, and for each $x\in \bar\Omega$
\[0\le m(x,t)\le \alpha \bar s(x) \text{ for all } t\le 0.\]
 Additionally, the following holds: Either $T_{\max}=\infty$ or $T_{\max}<\infty$, so that $\|u(\cdot,t)\|_{L^\infty(\Omega)}\to \infty $ as $t\to T_{\max}$.
\end{theorem}
    \begin{proof} We use a fixed-point argument to show existence and uniqueness. 
    Let $X=C^{1+\theta,(1+\theta)/2}(\overline{\Omega}\times [0,T_{\max}])$, and choose small enough $T\in (0,1)$. We define a closed bounded convex subset 
    \begin{equation*}
        K_T = \{ \tilde u\in X: \|\tilde u(\cdot,t)\|_{L^\infty(\overline{\Omega})}  \le M \text{ a.e. in } [0,T]\} 
    \end{equation*}
    For a given 
    $\Tilde{u}\in K_T$ we consider the auxiliary system
    \begin{equation}\label{eqn:aux}
\left\{\begin{aligned}
    &u_t = \nabla \cdot (\gamma(m)\nabla u) + \nabla\cdot (\gamma'(m)u\nabla m) +  g(s,\tilde u)u,  && \text{ in }\Omega \times (0,T)\\
    &m_t = (\alpha \bar s -  m) \tilde u-\mu m,\\
    &\nabla (\gamma(m)u) \cdot \textbf{n}=0 && \text{ on }\partial \Omega\times [0,T],
\end{aligned}\right.
\end{equation}
        Since $m$ can be expressed in the explicit form in terms of $\tilde u$:
        \begin{equation*}
            m(x,t) = e^{-(\int_0^t\tilde udt +\mu t)}m_0(x)+ \alpha \bar s(x)\int_0^t e^{-(\int_\sigma^t\tilde u dt+\mu(t-\sigma))}\tilde u(x,\sigma)d\sigma,
        \end{equation*} one can obtain the estimate for $m$:
        \begin{equation*}
            \|m\|_{X}\le C(\|\tilde u\|_{X} + \|m_0\|_{C^{1+\theta}(\overline{\Omega})}),
        \end{equation*} which results in $\nabla m\in L^{\infty}(\overline\Omega \times [0,T])$. By hypothesis \eqref{hyp: gamma}, $\gamma'(m)\nabla m \in L^{\infty}(\overline\Omega\times [0,T])$. Note that $g(x,\tilde u)\le C(1+M)$. 
        
        We claim that $m$ is uniformly bounded. Indeed, if $m\ge \alpha \bar s$, then
        \begin{equation*}
            m_t \le -\mu m\implies m(x,t)\le m(x,0)e^{-\mu t},
        \end{equation*} 
        and we have $0\le m(x,t) \le \max\{m_0(x),\alpha \bar s(x)\}.$ Since $\gamma$ is decreasing, there exists a constant $\gamma_0>0$ such that $\gamma(m)\ge \gamma_0$, which gives the uniform parabolicity to \eqref{eqn:aux}. 
        By Theorem V1.1 in \cite{Ladyzhenskaya1968}, there exists a unique solution $u$ to \eqref{eqn:aux} and 
        \begin{equation*}
            \|u\|_{C^{\theta_1,\theta_1/2}(\overline{\Omega}\times[0,T])}\le C
        \end{equation*} for some constant $C>0$. One can find
        \begin{equation*}
            \|u(\cdot,t)\|_{L^{\infty}(\Omega)}\le \|u_0\|_{L^{\infty}(\Omega)}+Ct^{\theta_1/2} 
        \end{equation*} for all $t\in (0,T)$.
        For sufficiently small $T$, we have $u(\cdot,t)\in L^\infty(\Omega)$. Due to the regularity of the initial condition and the parabolic regularity theory, $u\in X$, hence $u\in K_T$. We define a map $\Phi:K_T \to K_T$ by $u=\Phi(\tilde u)$.

       Following \cite{Choi2025DW, Jin2018}, the continuity and compactness of $\Phi$ are obtained by analogous arguments.
        By Schauder's fixed point theorem we thus obtain a local solution on $[0,T]$. Nonnegativity is followed from the maximum principle. Since $M$ depends only on the initial data, the construction can be iterated to extend the solution up to the maximal time $T_{\max}$. If $T_{\max}<\infty$, $\|u(\cdot,t)\|_{L^\infty(\Omega)}\to \infty$ as $t\to T_{\max}$. Otherwise, we can iterate the same process infinitely, so that $T_{\max}=\infty$.
        Uniqueness follows from the standard subtraction method in \cite{Choi2025DW,Jin2018}.
\end{proof}

\subsection{Population persistence and steady-state solutions}
In this subsection we show that the extinction equilibrium is unstable and that, under suitable conditions on $\gamma$, there exists a positive stable steady-state solution of \eqref{eqn:logistic}. Biologically, this corresponds to persistence of the population.

We now specialize on the classical logistic growth $g(s,u)=s-u$ and study the resulting steady-state structure. Our main questions are whether the population persists and whether the long-time spatial distribution is uniquely determined when individuals follow memory-based Fokker-Planck type diffusion.

\begin{lemma}[Unstable extinction equilibrium]
    Let $0<\theta<1$. Let $\Omega\subset \mathbb{R}^n$ be a bounded domain with $C^{2}$ boundary. Assume \eqref{hyp: gamma}, $s\in C^{2+\theta}(\overline{\Omega})$, and $s>0$ in $\overline \Omega$. Suppose that $g(s,u)=s-u$. Then the extinction steady-state solution $(u,m)=(0,0)$ of \eqref{eqn:logistic} is linearly unstable.
\end{lemma}
\begin{proof}
    Let $U=\gamma(m)u$. Then 
    $U_t = \gamma(m)u_t +\gamma'(m)um_t$, so
    \begin{equation}\label{eqn:modified}
        \left\{\begin{aligned}
            & U_t = \gamma(m)\Delta U +U\left(s(x)-\frac{U}{\gamma(m)}\right) + \gamma'(m)\frac{U}{\gamma(m)}\left((\alpha \bar s(x)-m)\frac{U}{\gamma(m)}-\mu m\right), && \text{ in }\Omega,\\
            & m_t = (\alpha \bar s(x) -m)\frac{U}{\gamma(m)} -\mu m,\\
            &\nabla U\cdot \textbf{n}=0, && \text{ on }\partial\Omega\\
            &U(x,0) = U_0(x),\,m(x,0) = m_0(x) && \text{ in } \Omega.
        \end{aligned}\right.
    \end{equation}  
    Linearizing the system around the extinction equilibrium $(0,0)$ and dropping all quadratic terms gives
    \begin{equation}\label{sys:linearized}
        \partial_t\begin{pmatrix}
            U\\m
        \end{pmatrix} = \mathcal{L}\begin{pmatrix}
            U\\m
        \end{pmatrix}:=\begin{pmatrix}
            \gamma(0)\Delta+s(x)&0\\ \frac{\alpha \bar s(x)}{\gamma(0)} &-\mu
\end{pmatrix}\begin{pmatrix}
            U\\m
        \end{pmatrix}
    \end{equation}
    Since $\mathcal{L}$ is lower-triangular in block form, its spectrum is simply $\sigma(\gamma(0)\Delta+ s(x))\cup\{-\mu\}$. 
    Since $s>0$, the principal eigenvalue $\lambda_1$ of $\gamma(0)\Delta +s(x)$ under Neumann boundary condition is positive. Indeed, the Rayleigh quotient is 
    \begin{equation}\label{lambda1}
        \begin{aligned}
            \lambda_1 &= \max_{\substack{\phi\in W^{1,2}(\Omega)\\ \phi\neq 0}} \frac{-\gamma(0)\int_\Omega |\nabla \phi|^2dx +\int_\Omega  s(x)\phi^2dx}{\int_\Omega \phi^2dx}\\
            &\ge  \frac{\int_\Omega s(x)\cdot 1^2dx}{\int_\Omega 1^2dx}= \frac{1}{|\Omega|}\int_\Omega  s(x)dx>0.
        \end{aligned}
    \end{equation} $U=0$ if and only if $u=0$ by positivity of $\gamma(m)$ due to boundedness of $m$. Therefore, $(u,m)=(0,0)$ is linearly unstable.
\end{proof}

We use the Schaefer's fixed point theorem to prove the existence of positive steady-state solutions of our model \eqref{eqn:logistic}. 

\begin{theorem}[Schaefer's Fixed Point Theorem \cite{EvansPDE}]\label{thm:Schaefer}Let $X$ be a Banach space, $K \subset  X$ be a nonempty closed convex set, and $T : K \rightarrow  K$ be continuous compact mapping. If a set$$\{ u \in  K : u = \lambda Tu \text{ for some } 0 < \lambda  \leq  1\}$$ is bounded, the map $T$ has a fixed point in $K$.
\end{theorem}

For the uniqueness of the steady-state, we use the following lemma.
\begin{lemma}\label{lem:BO}
Let $\Omega\subset\mathbb{R}^n$ be a bounded domain with $C^{2}$ boundary and let
$f:\Omega\times(0,\infty)\to\mathbb{R}$ be a Carath\'eodory function, i.e.,
$x\mapsto f(x,u)$ is measurable for every $u>0$ and $u\mapsto f(x,u)$ is
continuous for a.e.\ $x\in\Omega$.
Assume that $f(\cdot,u(\cdot))\in L^2$ and for a.e.\ $x\in\Omega$ the map
\[
u\mapsto \frac{f(x,u)}{u}
\quad\text{is strictly decreasing on }(0,\infty).
\]
Then the semilinear Neumann problem
\begin{equation}\label{eqn: BO_neu}
\left\{\begin{aligned}
&\Delta u + f(x,u)=0 && \text{in }\Omega,\\
&\nabla u\cdot \boldsymbol{n} = 0 &&  \text{on }\partial\Omega
\end{aligned}\right.
\end{equation}
admits at most one weak solution $u\in H^{1}(\Omega)$ satisfying $u>0$ a.e. in $\Omega$.
In particular, \eqref{eqn: BO_neu} has at most one positive classical solution
$u\in C^{2}(\bar\Omega)$.
\end{lemma}

\begin{proof}
In \cite{Brezis1986}, they showed the uniqueness of the weak solution to the Dirichlet boundary problem
\begin{equation*}
\left\{\begin{aligned}
&\Delta u + f(x,u)=0 && \text{in }\Omega,\\
&u = 0 &&  \text{on }\partial\Omega.
\end{aligned}\right.
\end{equation*}
Since the Dirichlet boundary condition was only used to cancel boundary terms in the
integration by parts in the paper, one can follow the similar argument to obtain the uniqueness of a positive weak solution to the Neumann boundary problem \eqref{eqn: BO_neu}. Since any $C^{2}(\overline\Omega)$ solution satisfies the weak
formulation, the uniqueness also holds in the class of positive classical
solutions.
\end{proof}

\begin{theorem}
    Let $0<\theta<1$ and $\mu>0$. Let $\Omega\subset \mathbb{R}^n$ be a bounded domain with $C^{2}$ boundary. Assume that \eqref{hyp: gamma} holds and $s\in C^{2+\theta}(\overline{\Omega})$ and $s>0$ in $\overline \Omega$. Then there exists a positive steady-state solution $(u,m)$ to \eqref{eqn:logistic} if
    \begin{equation}\label{exist cond:gamma}
     \sup_{z\in [0,M]}\left|\frac{\gamma'(z)}{\gamma(z)}\right|<\frac{4}{M},
    \end{equation} where $M=\alpha\|\bar s\|_\infty$. In addition, the solution exists uniquely if 
    \begin{equation}
        \label{uniq cond:gamma}
        \sup\limits_{[0,M]}\left|\frac{\gamma'(z)}{\gamma(z)}\right|\le \frac{\mu}{M\|s\|_\infty}
    \end{equation}
\end{theorem}
\begin{proof}
    The steady-state $(U,m)$ of \eqref{eqn:modified} satisfies $U_t=m_t=0$ and 
    \begin{equation}\label{eqn:SS logistic}
        \left\{\begin{aligned}
            & 0 = \Delta U +\frac{U}{\gamma(m)}\left(s(x)-\frac{U}{\gamma(m)}\right), && \text{ in }\Omega,\\
            & 0 = (\alpha \bar s(x) -m)\frac{U}{\gamma(m)} -\mu m, && \text{ in }\Omega\\
            &\nabla U\cdot \textbf{n}=0, && \text{ on }\partial\Omega
        \end{aligned}\right.
    \end{equation} From the second equation, we have $m=\alpha\bar s\frac{U}{U+\mu\gamma(m)}$. Define a map $F$ by $$F(x,m,U)=m-\frac{\alpha\bar s U}{U+\mu\gamma(m)}.$$ Then
    \begin{align*}
        \partial_m F(x,m,U) = 1 +\alpha\bar s \frac{U\mu\gamma'(m)}{(U+\mu\gamma(m))^2}\ge 1+\alpha\bar s \frac{\gamma'(m)}{4\gamma(m)}>0,
    \end{align*} where we have used $\frac{U}{(U+c)^2}\le \frac{1}{4c}$ for $c>0$ and the last inequality comes from \eqref{exist cond:gamma}. By the implicit function theorem, we can write $m=m(U)$ and get a semilinear equation for $U$: 
    \begin{equation*}
    \left\{\begin{aligned}
         & 0 = \Delta U +\frac{U}{\gamma(m(U))}\left(s(x)-\frac{U}{\gamma(m(U))}\right), && \text{ in }\Omega,\\
            &\nabla U\cdot \textbf{n}=0, && \text{ on }\partial\Omega.
    \end{aligned}\right.
    \end{equation*} 

    Let $X=C^{0,\theta}(\overline{\Omega})$ and let us define a closed convex subset $K\subset X$ by
    $$K=\{f\in X : f\ge 0 \text{ in } \overline{\Omega}\}.$$
    
    Choose $\tilde U\in K$.
    Denote $\kappa(\tilde U) = \gamma(m(\tilde U))$.
    Consider the auxiliary equation for $U$:
    \begin{equation}\label{eqn:SS aux}
    \left\{\begin{aligned}
         & 0 = \Delta U +\frac{U}{\kappa(\tilde U)}\left(s(x)-\frac{U}{\kappa(\tilde U)}\right), && \text{ in }\Omega,\\
            &\nabla U\cdot \textbf{n}=0, && \text{ on }\partial\Omega.
    \end{aligned}\right.
    \end{equation} 
    The standard comparison method ensures the existence of the unique solution $U$ to \eqref{eqn:SS aux}. Indeed, if we choose $0<\eps<(\min_{\overline{\Omega}} s(x))\gamma(M)$, then $\underline U=\eps$ satisfies
    \[\Delta \underline U + \frac{s}{\kappa(\tilde U)}\underline U - \frac{1}{\kappa(\tilde U)^2}\underline U^2 = \frac{s}{\kappa(\tilde U)}\eps - \frac{1}{\kappa(\tilde U)^2}\eps^2= \frac{s\eps}{\kappa(\tilde U)}\left(1-\frac{\eps}{s\kappa(\tilde U)}\right)\ge 0,  \] hence $\underline U=\eps$ is a subsolution. Similarly, a constant $\overline U = \|s\|_\infty\gamma(0)$ is a supersolution. By the standard iteration method, there exists a solution $U\in C^{2+\theta}(\overline{\Omega})$ to \eqref{eqn:SS aux}, and the solution is uniformly bounded by $\|s\|_\infty\gamma(0)$. Positivity of $U$ follows from the maximum principle and Hopf's lemma. Uniqueness of the positive $U$ is ensured due to the convexity of the logistic term in terms of $U$.  
    Define the mapping $T:K\to K$ by
     \begin{equation*}
        T(\tilde U) = U.
    \end{equation*} 

    \textbf{(Compactness of $T$)} Suppose that $B\subset K$ is bounded in $X$.
    We claim that $T(B)$ is precompact in $K$. Pick $\tilde u\in B$. Then $$T(\tilde U) =U \le \overline{U}=\|s\|_\infty\gamma(0).$$ Choose $p>n$. The Neumann $W^{2,p}$ estimate gives
    \begin{align*}
        \|U\|_{W^{2,p}(\Omega)}&\le C\left(\|U\|_{L^{p}(\Omega))}+\left\|\frac{s}{\kappa(\tilde U)} U - \frac{1}{\kappa(\tilde U)^2}U^2\right\|_{L^{p}(\Omega)}\right)\\
        &\le C=C(\Omega, s,\gamma).
    \end{align*}  By using the Sobolev embedding, we have for $\delta=1-\frac{n}{p}\in (0,1)$ 
    \[\|U\|_{C^{1+\delta}(\overline\Omega)}\le C=C(\Omega,s,\gamma,p). \] Therefore, $U$ is also uniformly bounded in $C^{0,\theta}(\overline{\Omega})$. 
    The Neumann Schauder estimate yields
    \begin{align}\label{ineq:Neumann Schauder}
        \|U\|_{C^{2+\theta}(\overline \Omega)}\le C\left(\|U\|_{C(\overline\Omega))}+\left\|\frac{s}{\kappa(\tilde U)} U - \frac{1}{\kappa(\tilde U)^2}U^2\right\|_{X}\right)\\
        \le C_1(1+[\tilde U]_{X}),
    \end{align} where we have used the $C^1$ regularity of $\gamma$ and the $C^{0,\alpha}$ regularity of the map  $m=m(\tilde U) = \alpha\bar s \frac{\tilde U}{\tilde U+\mu}$. Since $\tilde U\in B$, $\|U\|_{C^{2+\theta}(\overline \Omega)}$ is uniformly bounded when $\tilde U\in B$. By the Sobolev embedding, $T(B)$ is relatively compact in $K$. 
    
    \textbf{(Continuity of $T$)} Suppose that $\tilde U_n\to \tilde U$ in $K$ as $n\to\infty$. Due to the regularity of $\kappa$, one can deduce that $\kappa(\tilde U_n)\to\kappa(\tilde U)$ in $X$ as $n\to\infty$. Note that $\tilde U_n$ is bounded in $K$. The Neumann Schauder estimate \eqref{ineq:Neumann Schauder} gives the uniform bound of $T(\tilde U)$ in $C^{2,\theta}(\overline{\Omega})$. Let $0<\beta <\theta$. Then by the Sobolev embedding $C^{2,\theta}(\overline{\Omega})\hookrightarrow C^{2,\beta}(\overline{\Omega})$, we can choose a convergent subsequence $\{v_{n_k}\}_k$ of $v_{n} = T(\tilde U_n)$, say $v_{n_k}\to \hat U$ in $C^{2,\beta}(\overline{\Omega})$ as $n\to \infty$. Then $\Delta \tilde U_{n_k}\to \Delta \hat U$ in $C(\overline{\Omega})$, hence
    $\hat U$ satisfies \eqref{eqn:SS aux}. On the other hand, $T(\tilde U)$ also solves \eqref{eqn:SS aux}. The uniqueness of its solution concludes that $T(\tilde U)=\hat U=\lim_{n\to\infty}T(\tilde U_n)$ in $X$.

    \textbf{(Boundedness of $S$ in $K$)} We define a subset
    $S\subset K$ by
    \begin{equation*}
        S = \{U\in K: U=\lambda T(U) \text{ for some }\lambda \in (0,1]\}.
    \end{equation*} We claim that $S$ is bounded in $K$.
    Since $U=\lambda T(U)$ for $U\in S$ for $\lambda\in (0,1]$, it suffices to prove that $V:=T(U)$ is uniformly bounded in $K$.
     We apply the comparison principle and the Neumann $W^{2,p}$ estimate as we did above:
    \begin{align*}
        \|V\|_{W^{2,p}(\Omega)}&\le C\left(\|U\|_{L^{p}(\Omega))}+\left\|\frac{s}{\kappa(U)} V - \frac{1}{\kappa(U)^2}V^2\right\|_{L^{p}(\Omega)}\right)
    \end{align*}  The comparison principle gives $V\le \overline{U}=\|s\|_\infty\gamma(0)$, so $\frac{V}{\kappa(U)}\le \frac{\overline{U}}{\gamma(M)}$. Therefore, $\|V\|_{W^{2,p}(\Omega)}$ is uniformly bounded in $W^{2,p}(\Omega)$. By applying the Sobolev embedding theorem again, we obtain the boundedness of $V$ in $K$.

    We apply the Schaefer's fixed point theorem \ref{thm:Schaefer} to obtain a fixed point $U\in K$ such that $U=T(U)$. Denote $u=U/\gamma(m(U))$. Then $u$ is positive and solves 
    \begin{equation*}
    \left\{\begin{aligned}
        &0=\Delta(\gamma(m)u) + u(s-u) && \text{ in }\Omega\\
        &0=(\alpha\bar s-m)u-\mu m && \text{ in }\Omega\\
        &\nabla(\gamma(m)u)\cdot \boldsymbol{n}=0 && \text{ on }\partial\Omega.
    \end{aligned}\right.
    \end{equation*}

    \textbf{(Uniqueness)} Since \eqref{eqn:SS logistic} is of the form of the semilinear elliptic equation \eqref{eqn: BO_neu}, it suffices to show for the uniqueness that the map
    \begin{equation*}
        U\mapsto \frac{1}{\gamma(m(U))}\left(s-\frac{U}{\gamma(m(U))}\right)
    \end{equation*} is strictly decreasing. Note that $U=\gamma(m(u))u$ and 
    \begin{equation*}
        \frac{\partial U}{\partial u} = \gamma(m(u)) + \gamma'(m(u))\frac{\alpha\bar s \mu u}{(u+\mu)^2}\ge \gamma(m(u)) + \gamma'(m(u))\frac{\alpha\bar s }{4}>0
    \end{equation*} due to the inequality $\frac{u\mu}{(u+\mu)^2}\le \frac{1}{4}$ and \eqref{exist cond:gamma}. Therefore, it suffice to show that the map
    \begin{equation*}
        u\mapsto \frac{s-u}{\gamma(m(u))}
    \end{equation*} is strictly decreasing. The derivative of the map with respect to $u$ is given by
    \begin{equation*}
        \frac{1}{\gamma(m(u))^2}\left( -\gamma(m(u))-(s-u)\gamma'(m(u))\frac{\alpha \bar s \mu}{(u+\mu)^2}\right).
    \end{equation*}
    Equation \eqref{uniq cond:gamma} yields
    \begin{align*}
        -\gamma(m) -\gamma'(m)\frac{\alpha \bar s s\mu}{(u+\mu)^2}&\le -\gamma(m) +\left(\gamma(m)\frac{\mu}{M \|s\|_\infty}\right)\frac{\alpha \bar s s\mu}{(u+\mu)^2}\\
        &=\gamma(m)\left(-1 + \frac{\alpha\bar s s \mu^2}{M\|s\|_\infty (u+\mu)^2}\right)\\
        &<0
    \end{align*}
    which concludes the uniqueness of the positive solution to \eqref{eqn:SS logistic}.

\end{proof}

\begin{remark} 
    The condition \eqref{uniq cond:gamma} on $\gamma$ is a technical sufficient condition used to obtain uniqueness; the argument is not expected to be optimal, and a sharper proof might establish uniqueness of the steady state under weaker assumptions. Nevertheless, it is noteworthy that we can still guarantee a unique positive steady-state solution in the presence of the nonlinear diffusion term $\Delta(\gamma(m)u)$ driven by the cognitive map. This suggests that the uniqueness result  \eqref{eqn:constantgamma} in \cite{Cantrell1989} could potentially be extended to the diffusive--logistic model with memory-based diffusion considered here.
\end{remark}

\subsection{Lingering phenomenon with population growth}
We have analytically investigated the population persistence via existence theory and long-time behavior analysis. In this subsection, we present a partial analytical result of the lingering phenomenon for the model with logistic growth \eqref{eqn:logistic}. Additionally,  we observe the lingering through numerical simulation results of the population growth model \eqref{eqn:logistic}.

Let $\Omega = (-\ell,\ell)$ and $s$ is even in $(-\ell,\ell)$ and decreasing in $(0,\ell)$. 
Let $Q$ be a reflection operator centered at $x=0$. Note that $Qs =s$, so $\bar Qs = \bar s$. Then 
\begin{align*}
    &Q\left[(\gamma(m)u)'' +u(s-u)\right]= (\gamma(Qm)Qu)'' +Qu(s-Qu), \text{ and }\\ 
    & Q[(\alpha\bar s -m)u -\mu m] = (\alpha \bar s -Qm)Qu -\mu Qm.
\end{align*} Therefore, $(u,m)$ and $(Qu,Qm)$ satisfy the steady-state problem 
\begin{equation}\label{eqn: SteadyState}
    \left\{\begin{aligned}
        &0 = (\gamma(\phi)v)''+ v (s - v)\\
        &\phi = \alpha \bar s \frac{v}{v+\mu}\\
        & (\gamma(\phi)v)'(\pm \ell) =0.
    \end{aligned}\right.
\end{equation} Due to the uniqueness of the steady-state solution, it follows that $u=Qu$ and $m=Qm$, that is, $u$ and $m$ are symmetric at $x=0$. 

For small $\alpha$, we will prove that the steady-state solution of \eqref{eqn: SteadyState} behaves similarly to the steady state in the constant-$\gamma$ case (i.e. the steady-state solution of \eqref{eqn:constantgamma}). The following Lemma~\ref{lem: even SS} provides a key monotonicity property of the steady-state solution to \eqref{eqn:constantgamma} (constant $\gamma$) near the boundary. A similar result was obtained by \cite{Gidas1979} under Dirichlet boundary conditions. In the Neumann case considered here, we use a different argument to establish the monotonicity.

\begin{lemma}\label{lem: even SS}
Let $\gamma_0,\ell>0$, and let $s\in C^1([-\ell,\ell])$ be even and strictly decreasing on $(0,\ell)$.
Suppose that $v\in C^2([-\ell,\ell])$ is the unique positive solution of
\begin{equation}\label{eqn: semilinear elliptic}
\gamma_0v''+v(s-v)=0 \quad\text{in }(-\ell,\ell),\qquad v'(-\ell)=v'(\ell)=0.
\end{equation}
Then $v$ is even and nonincreasing on $[0,\ell]$. In particular, $v$ attains its maximum at $x=0$.
\end{lemma}

\begin{proof}
We first show that $v$ is even. Define $\tilde v(x):=v(-x)$. Since $s$ is even, $\tilde v$ satisfies
\[
\gamma_0\tilde v''+\tilde v(s-\tilde v)=0 \quad\text{in }(-\ell,\ell),\qquad 
\tilde v'(-\ell)=\tilde v'(\ell)=0,
\]
and $\tilde v>0$ on $[-\ell,\ell]$. By the uniqueness of the positive solution, we conclude that
$v\equiv \tilde v$, hence $v$ is even.

Next, we claim that there exists $x\in(0,\ell)$ such that $v'(x)\le 0$.
Indeed, suppose by contradiction that $v'(x)>0$ for all $x\in(0,\ell)$, so that $v$ is strictly increasing on $(0,\ell)$.
Since $v$ is even, it follows that $v$ attains its minimum at $x=0$, hence $v'(0)=0$ and $v''(0)\ge 0$.
Evaluating \eqref{eqn: semilinear elliptic} at $x=0$ gives
\[
\gamma_0 v''(0)=-v(0)\bigl(s(0)-v(0)\bigr)\ge 0,
\]
and therefore $v(0)\ge s(0)$.
Similarly, since $v$ attains its maximum at $x=\ell$ and $v'(\ell)=0$, we have $v''(\ell)\le 0$, and thus
$v(\ell)\le s(\ell)$.
This is impossible because
\[
v(\ell)\le s(\ell)<s(0)\le v(0),
\]
while $v(0)$ is the minimum of $v$. 

Now set $w:=v'$ and define
$$
E:=\{x\in(0,\ell): w(x)>0\}.
$$
Then $E$ is open by continuity of $w$. If $E$ were nonempty, take any $x_1\in E$ and let $(a,b)\subset (0,\ell)$
be the connected component of $E$ containing $x_1$. Then $w>0$ on $(a,b)$ and $w(a)=w(b)=0$.
Hence $w$ attains its (positive) maximum at some $x_*\in(a,b)$, and thus
\[
w'(x_*)=0,\qquad w''(x_*)\le 0.
\]
Since $s\in C^1$ and $v\in C^2$, we have $v''\in C^1$, and differentiating \eqref{eqn: semilinear elliptic} yields
\begin{equation}\label{eqn: semilinear w}
\gamma_0 w''+(s-2v)w=-vs' \qquad \text{in }(0,\ell).
\end{equation}
Moreover, $v''(x_*)=w'(x_*)=0$, so evaluating \eqref{eqn: semilinear elliptic} at $x=x_*$ gives
\[
0=\gamma_0 v''(x_*)=-v(x_*)\bigl(s(x_*)-v(x_*)\bigr),
\]
and since $v(x_*)>0$, we obtain $v(x_*)=s(x_*)$.

Substituting $v(x_*)=s(x_*)$ into \eqref{eqn: semilinear w}, we get
\[
\gamma_0 w''(x_*)=s(x_*)\bigl(w(x_*)-s'(x_*)\bigr).
\]
Because $w(x_*)>0$ and $s'(x_*)\le 0$ (as $s$ is decreasing on $(0,\ell)$), the right-hand side is strictly positive.
Thus $w''(x_*)>0$, contradicting $w''(x_*)\le 0$.
Therefore, $E$ must be empty, and hence $v'(x)=w(x)\le 0$ for all $x\in(0,\ell)$.

Since $v$ is even and nonincreasing on $[0,\ell]$, it attains its maximum at $x=0$.
\end{proof}

We prove a no-lingering result for small $\alpha$.

\begin{theorem} \label{thm:small alpha}For $0<\theta<1$, let $\gamma\in C^{2+\theta}([0,\infty))$. Assume that $s\in C^{1+\theta}([-\ell,\ell])$ is even and positive, decreasing on $(0,\ell)$. Let $\alpha,\mu\ge 0$. Let $u(x;\alpha)$ be a solution of \eqref{eqn: SteadyState}. If $s''(0)<0$, then there exists small $\alpha^*>0$ such that $u(0;\alpha)< s(0)$ for all $0<\alpha\le\alpha^*$. In addition, 
\begin{equation*}
    \int_{-\ell}^{\ell}u(x;\alpha)dx > \int_{-\ell}^{\ell}s(x)dx
\end{equation*} for such $\alpha$.
    
\end{theorem}
\begin{proof}
Suppose that $s''(0)<0$.
    Let $\alpha =0$ and denote the corresponding positive solution $u_0$, that is, 
    \begin{equation}\label{eqn: zeroalpha}
    \left\{\begin{aligned}
        &0 = \gamma(0)u_{0}'' + u_{0} (s - u_{0})\\
        & u_{0}'(\pm \ell) =0.
    \end{aligned}\right.
\end{equation} Then $u_{0}$ is even, positive, and decreasing in $(0,\ell)$ by Lemma \ref{lem: even SS}. Since $u_{0}$ has the maximum at $x=0$, 
\[u_{0}(0)(s(0)-u_{0}(0)) = -\gamma(0)u_{0}''(0)\ge 0 \implies u_{0}(0)\le s(0) .\] If $u_0(0)=s(0)$, then $u_0''(0)=0$. 
Let $z=u_0-s$. 
Then $z$ satisfies $$\gamma_0z''-u_0z=-\gamma_0s'' \text{ and } z(0)=0,\,z'(0)=0,$$ hence $z''(0) = -s''(0)>0$. We obtain that there exists a small $\delta>0$ such that $z(x)>0$ for all $x\in(0,\delta)$, hence $u_0(x)>s(x)$ for all $x\in(0,\delta)$. From \eqref{eqn: zeroalpha}, for $x\in(0,\delta)$, 
$$u_0''(x) = \frac{u_0(x)}{\gamma_0}(u_0(x)-s(x))>0.$$ However, since $u_0'(0)=0$ and $u_0$ is decreasing, for $x\in (0,\delta)$, $$u''(\xi) = \frac{u_0'(x)-u_0'(0)}{x-0}=\frac{u_0'(x)}{x}\le 0 \quad \text{ for some } \xi\in(0,x).$$ This is a contradiction. Therefore, $u_0(0)<s(0)$.
 The last conclusion for $\alpha=0$ follows from \cite[Theorem 1.2]{Lou2006}. 
 
Now it suffices to show that for sufficiently small $\alpha$, $u_\alpha\to u_0$ in $C^{2+\theta}(\overline{\Omega})$ as $\alpha\to 0$. For $\alpha>0$, define an operator
\begin{equation*}
    F(\alpha,v) = \begin{pmatrix}
        F_1(\alpha,v)\\F_2(\alpha,v)\\F_3(\alpha,v)
    \end{pmatrix}=\begin{pmatrix}
        (\gamma(\phi)v)''+v(s-v)\\(\gamma(\phi)v)'(-\ell)\\(\gamma(\phi)v)'(\ell)
    \end{pmatrix} 
\end{equation*} mapping $\mathbb{R}^+\cup\{0\}\times C^{2+\theta}([-\ell,\ell])\to C^\theta([-\ell,\ell])\times \mathbb R\times \mathbb R$, where $\phi=\alpha\bar s \frac{v}{v+\mu}$. Consider the Frechet derivative of $F$ with respect to $v$ at $(0,u_0)$
\begin{equation*}
    D_vF(0,u_0)(v)=\begin{pmatrix}
        \gamma(0)v''+(s-2u_0)v\\ \gamma(0)v'(-\ell)\\ \gamma(0)v'(\ell)
    \end{pmatrix} 
\end{equation*} 
Standard elliptic theory (\cite{Cantrell1989,GilbargTrudinger,Lou2006}) implies that $D_vF(0,u_0)$ is a bounded Fredholm operator of index zero from $C^{2+\theta}([-\ell,\ell])\to C^\theta([-\ell,\ell])\times \mathbb R\times \mathbb R$. Moreover, any positive solution $u_0$ of \eqref{eqn: zeroalpha} is nondegenerate, that is, \begin{equation*}
   L\psi:= \gamma(0)\psi''+(s-2u_0)\psi=0, \quad \psi'(\pm\ell)=0
\end{equation*} admits only the trivial solution $\psi\equiv 0$ (i.e. $0\notin \sigma (L)$). Thus, $D_vF(0,u_0)$ is an isomorphism. Therefore, by the inverse function theorem, there exists $\alpha^*>0$ and a $C^1$ branch $\alpha\mapsto u_\alpha$ such that
$F(\alpha,u_\alpha)=0$ for $0\le \alpha<\alpha^*$ and $u_\alpha\to u_0$ in $C^{2+\theta}([-\ell,\ell])$ as $\alpha\to 0$.
\end{proof}

\begin{figure}[ht]
    \centering
    \includegraphics[width=0.4\linewidth]{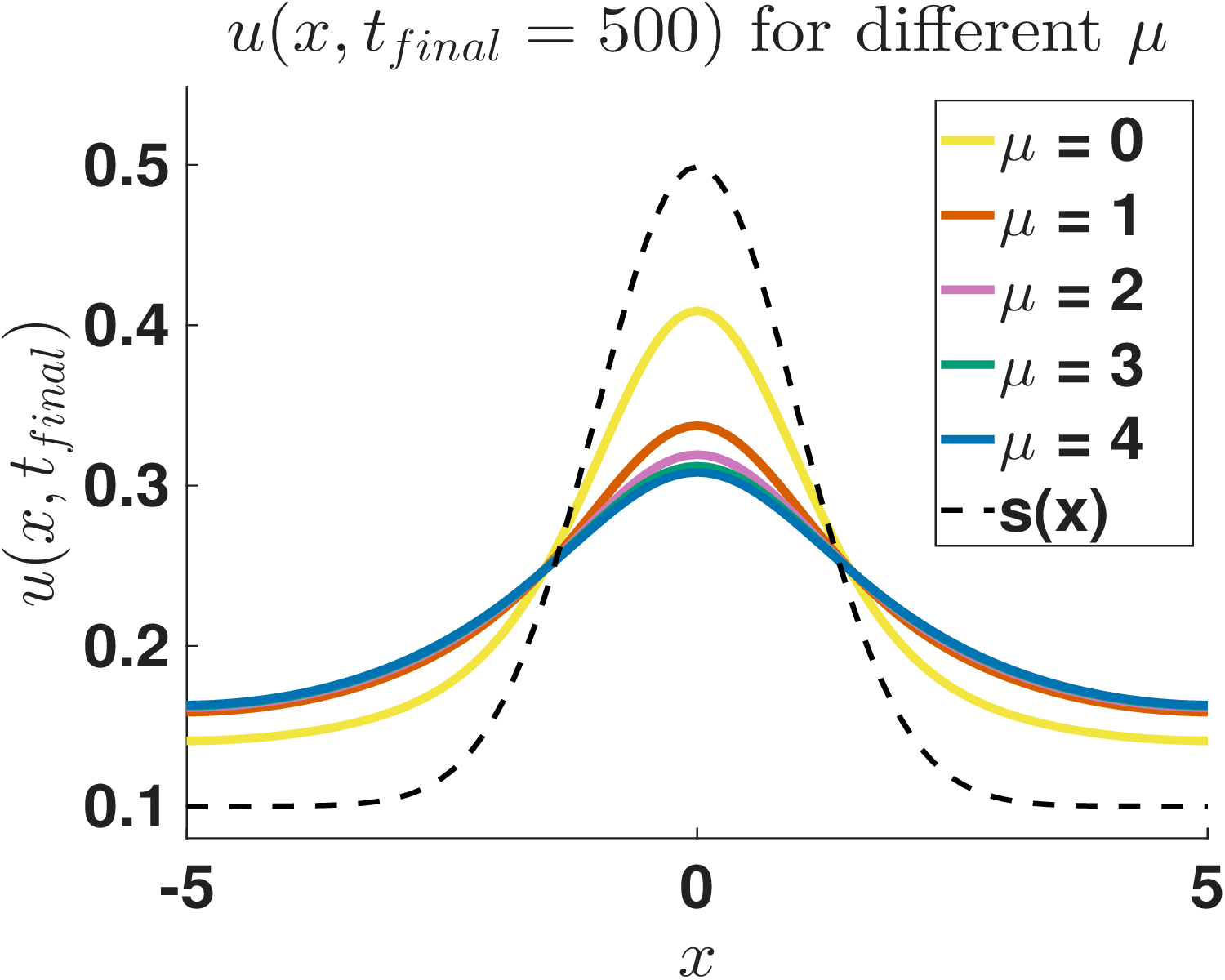}~
    \includegraphics[width=0.4\linewidth]{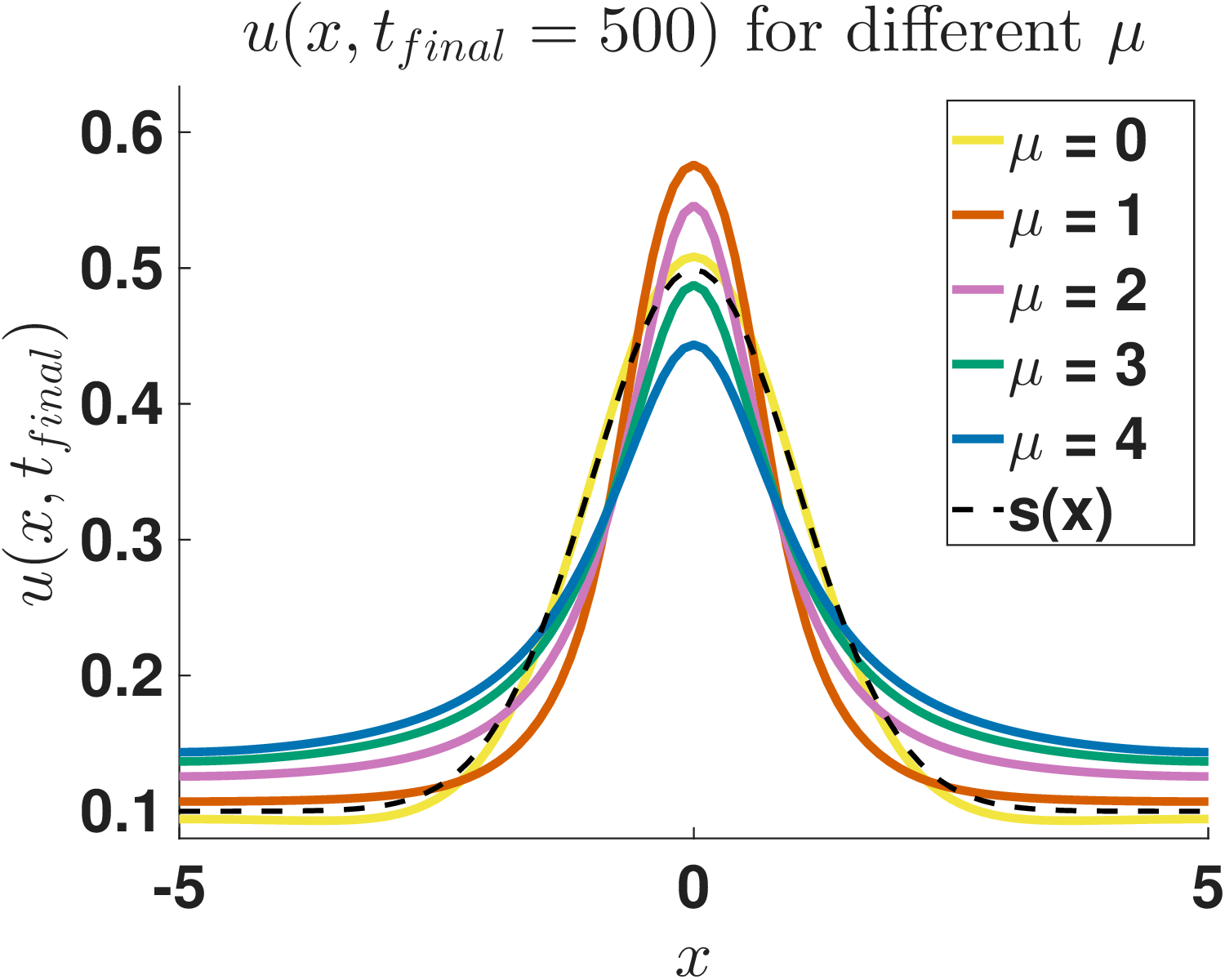}
    \caption{Snapshots of numerical solutions $u(x,t_{\mathrm{final}}=500)$ to the logistic model \eqref{eqn:logistic} with different values of the forgetting rate $\mu$ and the unimodal resource distribution $s(x)=\frac{1}{\sqrt{2\pi}}e^{-x^{2}/2}+0.1$. 
The dashed curve shows $s(x)$. 
The left panel corresponds to the learning rate $\alpha=1$ and the right panel to $\alpha=10$. 
For the larger learning rate the population exhibits stronger concentration near the resource peak for intermediate values of $\mu$, whereas very large $\mu$ leads to flatter profiles.
}
    \label{fig:lingering_logistic}
\end{figure}
Theorem~\ref{thm:small alpha} shows that, for sufficiently small learning rate $\alpha$, the population profile at the best resource location remains strictly below the carrying capacity $s(0)$. The diffusion driven by the cognitive map still smooths out the density near $x=0$. The numerical experiments in Figure~\ref{fig:lingering_logistic} show that, when $\alpha$ is large, the dependence of $u$ on the forgetting rate $\mu$ is non-monotonic, as in the purely movement-based model of Section~\ref{sec:lingering from dispersal}. Intermediate values of $\mu$ produce the strongest concentration near the resource peak, whereas very small or very large $\mu$ lead to flatter profiles. Thus the lingering phenomenon appears to persist when logistic growth is included, even though our current analytical results cover only the small-$\alpha$ regime.

\subsection{Memory regimes and population size trade-offs}\label{sec:comparison}
When the random diffusive--logistic model \eqref{eqn:constantgamma} is considered, that is, when $\gamma(z)\equiv \gamma>0$ for some positive constant $\gamma$, it is well known that if $s(x)$ is non-constant, bounded, measurable, and nonnegative for all $x\in \overline\Omega$, then the positive steady-state solution $u$ satisfies
\begin{equation}\label{totalu>totals}
    \int_{\Omega} u(x)\,dx > \int_\Omega s(x)\,dx
\end{equation}
for all $\gamma>0$ \cite[Theorem 1.2]{Lou2006}. 
 When it comes to our cognitive movement model \eqref{eqn:logistic}, the inequality \eqref{totalu>totals} holds for sufficiently small $\alpha$ as we have seen in Theorem \ref{thm:small alpha}. However, the inequality does not necessarily hold for large $\alpha$.

In this subsection, we quantify how memory parameters shape both lingering and population size through numerical simulations. 
Let $\Omega=[-5,5]$. 
We consider the unimodal function $s(x)$ given by
\[
s(x) = \frac{1}{\sqrt{2\pi}}e^{-5x^2}+0.1,
\]
which has total mass $\int_{-5}^5 s(x)\,dx \approx 1.3162$. 
We choose the initial condition for the following simulations to be $u_0=s$.
We measure lingering strength by the maximum value $\max_{x\in\overline\Omega} u(x,t_{\mathrm{final}})$. 
We also measure population performance by the total mass $\int_{-5}^{5} u(x,t_{\mathrm{final}})\,dx$. 
Our goal is to understand how these quantities depend on the learning rate $\alpha$ and the forgetting rate $\mu$ and to identify memory regimes that produce strong lingering but also support a large total population. 
To illustrate these effects, we consider representative choices of $\alpha$ and $\mu$ and compare the resulting peak density and total mass. 

\begin{figure}[ht]
    \centering
    
    \includegraphics[width=0.33\linewidth]{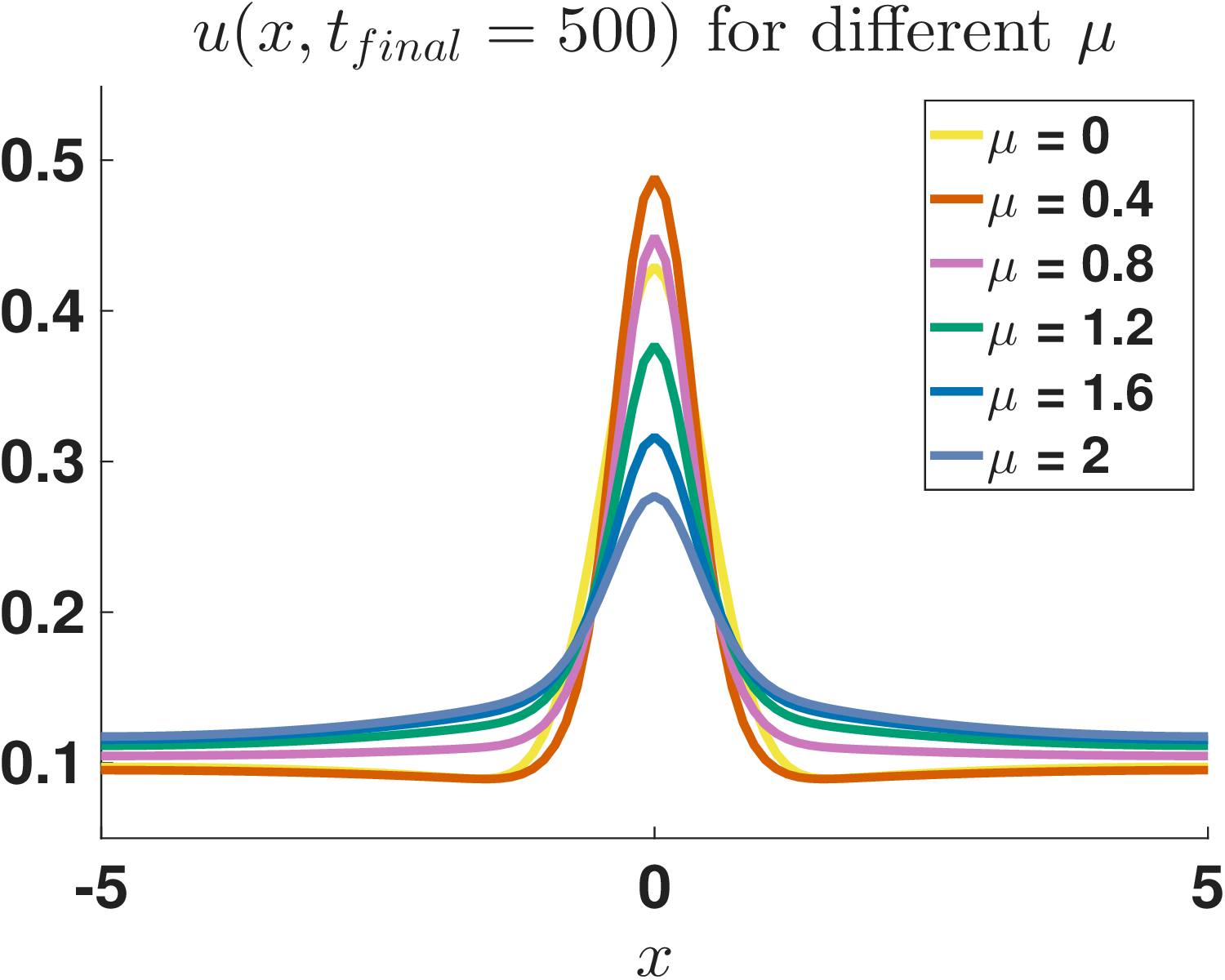}~\includegraphics[width=0.33\linewidth]{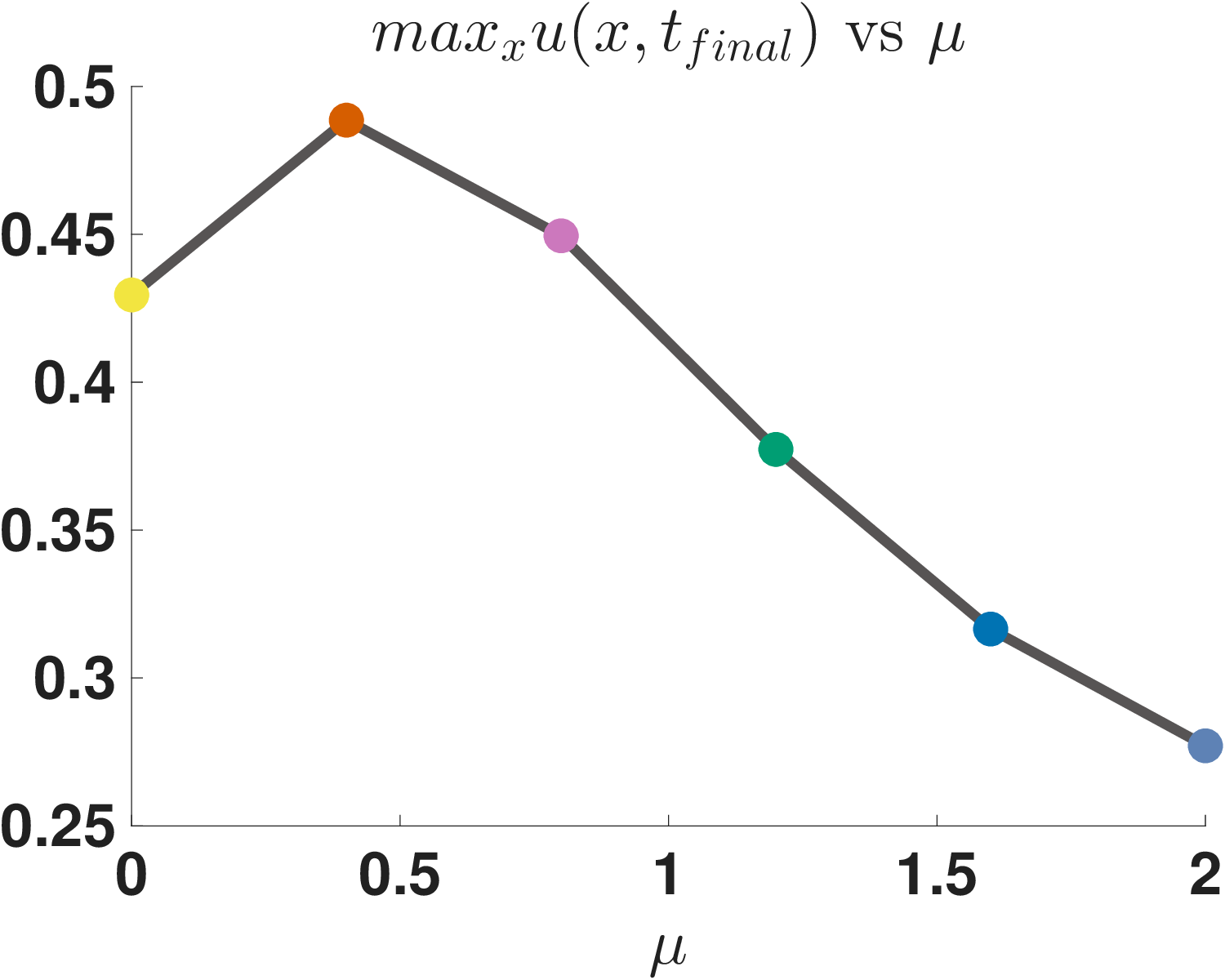}~\includegraphics[width=0.33\linewidth]{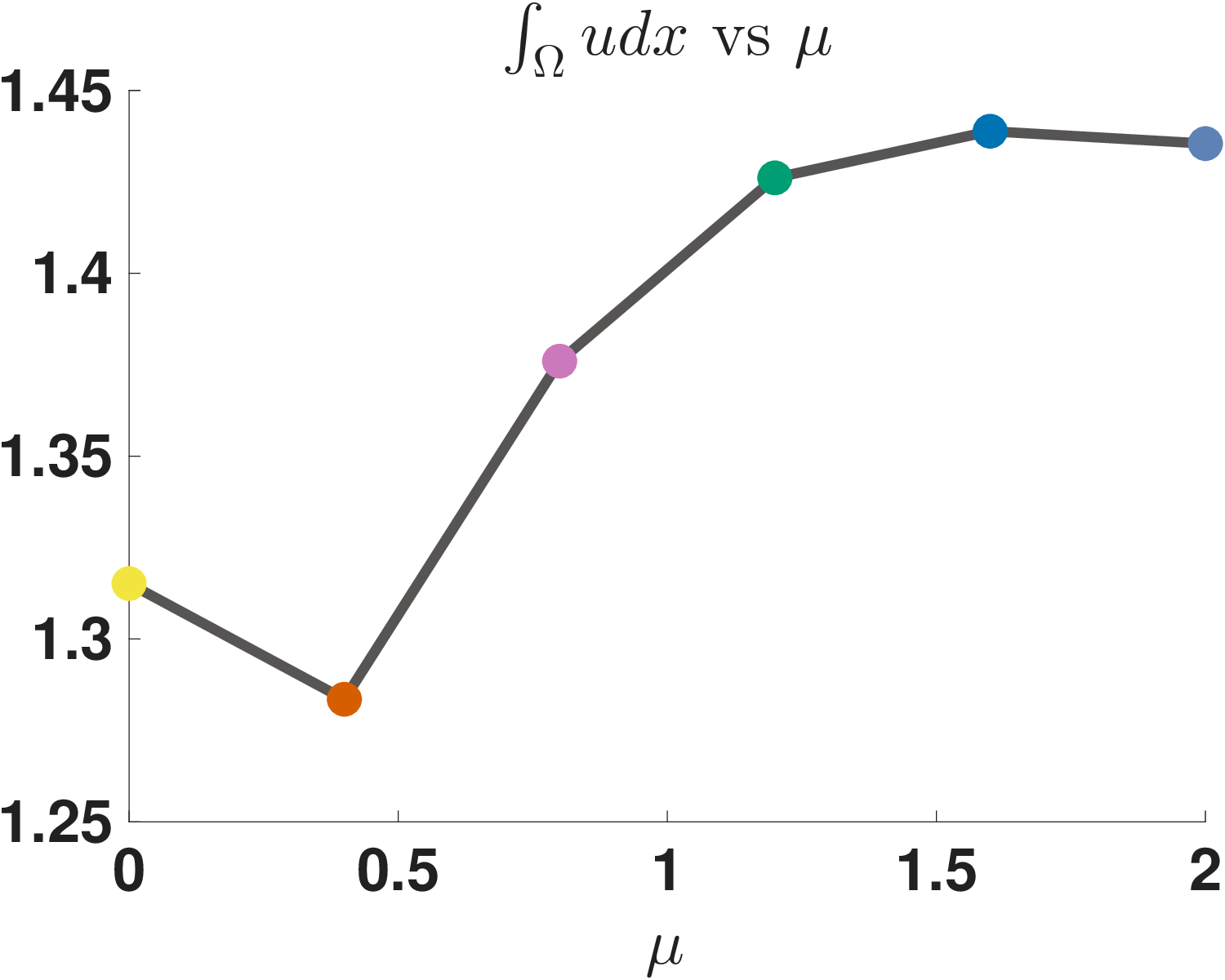}\\
    
    \vspace{0.2cm}
    
    \includegraphics[width=0.33\linewidth]{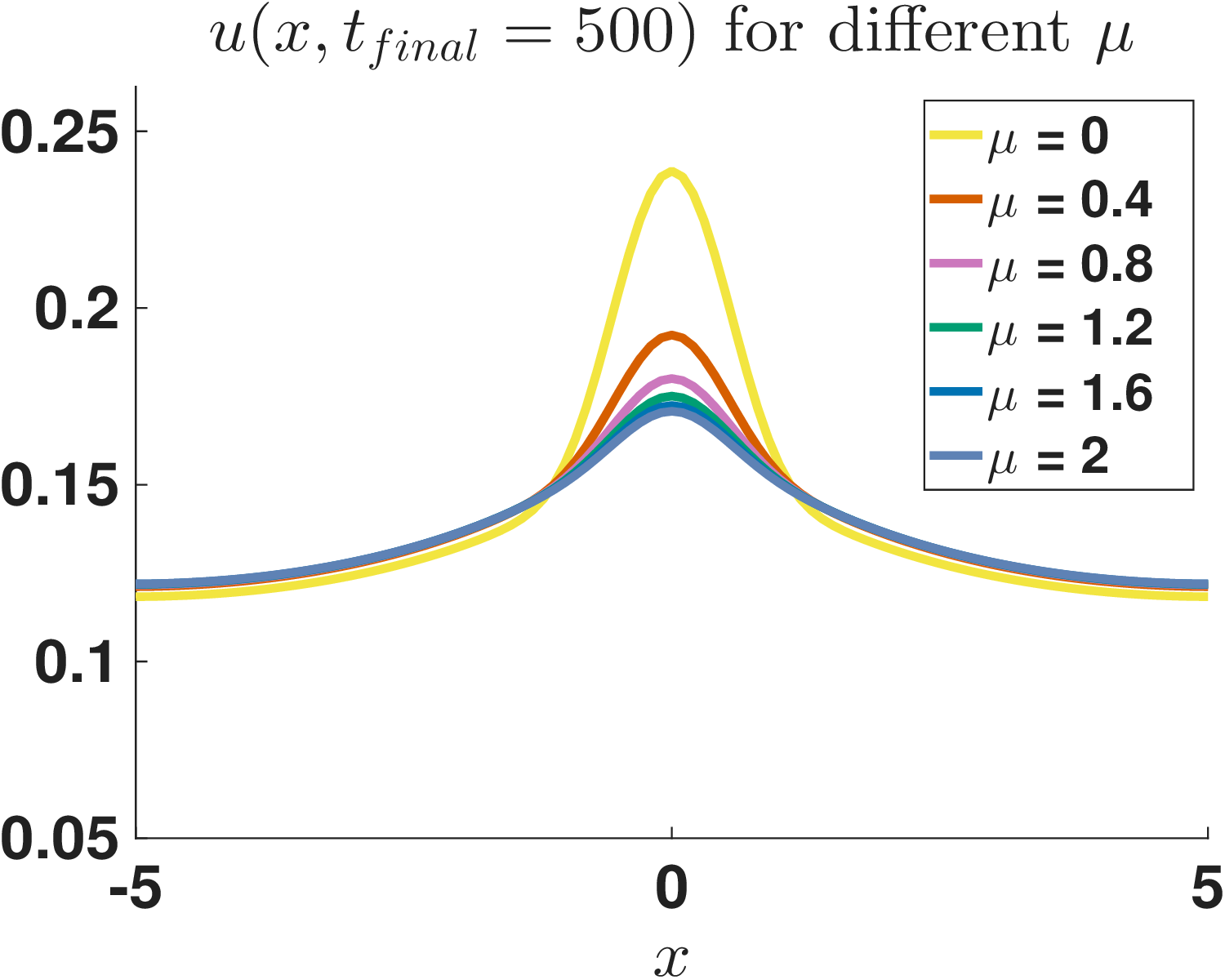}~\includegraphics[width=0.33\linewidth]{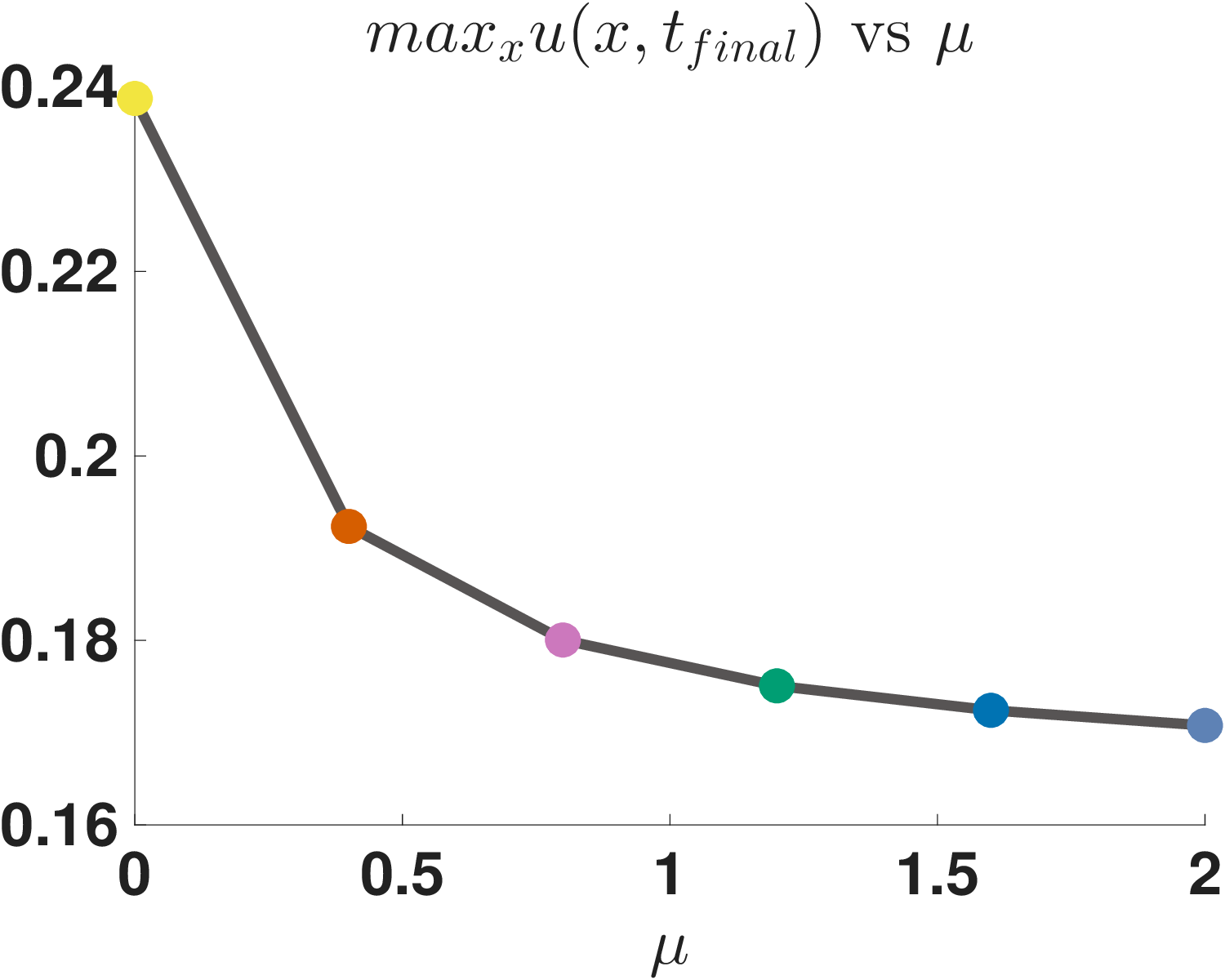}~\includegraphics[width=0.33\linewidth]{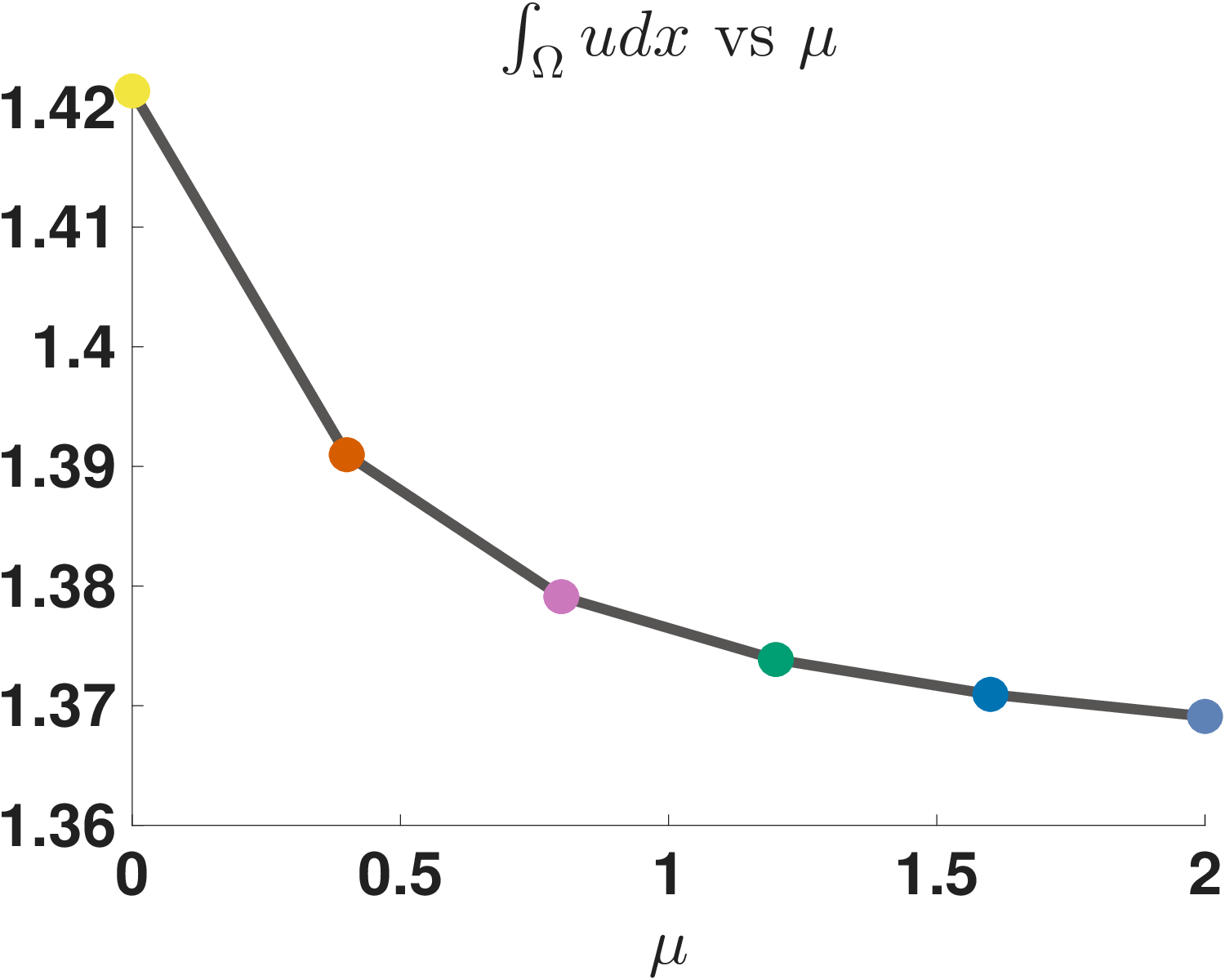}
    
    \caption{Numerical simulation results for the logistic growth model \eqref{eqn:logistic} with $\Omega=[-5,5]$ and the unimodal resource distribution $s(x)=\frac{1}{\sqrt{2\pi}}e^{-5x^2}+0.1$, whose total mass is $\int_{-5}^5 s(x)\,dx \approx 1.3162$. The forgetting rate $\mu$ varies and $t_{\mathrm{final}}=500$.  Left column: spatial profiles $u(x,t_{\mathrm{final}})$ for different $\mu$.  Middle column: peak density $\max_{x\in\Omega}u(x,t_{\mathrm{final}})$ as a function of $\mu$.  Right column: total population $\int_{-5}^5 u(x,t_{\mathrm{final}})\,dx$ as a function of $\mu$.  The first row corresponds to a large learning rate $\alpha=10$ and the second row to $\alpha=1$. For $\alpha=10$ the peak density is maximal near $\mu\approx 0.4$ whereas the total population is maximal near $\mu\approx 1.6$, so the dependence of $\int_{-5}^5 u(x,t_{\mathrm{final}})\,dx$ on $\mu$ is non-monotone and the inequality $\int_{-5}^5 u(x,t_{\mathrm{final}})\,dx<\int_{-5}^5 s(x)\,dx$ occurs only near the strongest lingering regime.}
    \label{fig:totalu}
\end{figure}
In Figure~\ref{fig:totalu}, we numerically observe that memory parameters control both how sharply the population lingers near the resource peak and how large the total population can grow. 
For a large learning rate $\alpha=10$, our simulations show that the strongest lingering occurs at $\mu\approx 0.4$, whereas the total mass is minimized at $\mu\approx 0.4$. The total mass is maximized near $\mu\approx 1.6$. If $\mu$ is even larger, total mass decreases. (See Figure~\ref{fig:comparison large mu}.)
Thus, the memory regime that maximizes lingering does not maximize population size. 
Instead, it actually reduces population size once linear growth and intraspecific competition (logistic growth) are taken into account. 

 Our simulations in Figure \ref{fig:totalu} show that
\[
\int_{-5}^5 u(x,t_{\mathrm{final}})\,dx < \int_{-5}^5 s(x)\,dx
\]
occurs only in a narrow neighbourhood of the values of the forgetting rate $\mu$ that produce the strongest lingering. 
This behaviour is also studied from recent work on resource-dependent dispersal, where non-constant dispersal rates can lead to a total population smaller than the resource distribution when the motility function $\gamma$ decays sufficiently fast \cite{Tang2023}. 
In our model, the possibility of the reverse inequality to \eqref{totalu>totals} arises from how quickly individuals learn and forget, that is, from the combined effect of the parameters $\alpha$ and $\mu$. Consequently, lingering is not a beneficial strategy for populations.

\begin{figure}[ht]
    \centering
    \includegraphics[width=0.4\linewidth]{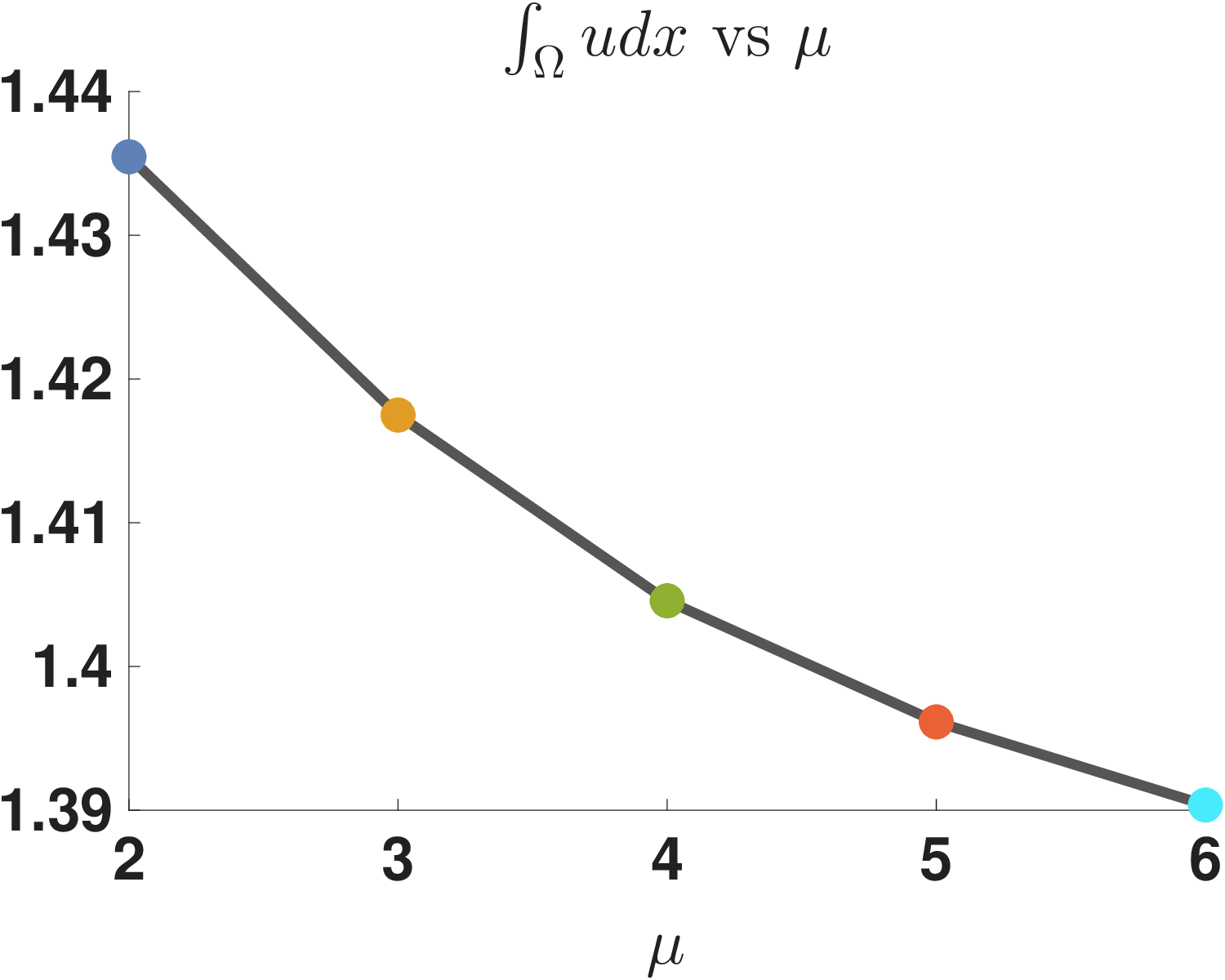}
    \caption{The total mass $\int_{-5}^5udx$ vs $\mu=2,3,4,5,$ and $6$ when $\alpha=10$.}
    \label{fig:comparison large mu}
\end{figure}

\section{Discussion}\label{sec:discussion}
In this work, we propose a population model with the Fokker--Planck type diffusion in which movement depends on a cognitive map $m(x,t)$  constructed from nonlocal perception. We observed that nonlocal perception can cause  spatial heterogeneity in cognitive maps and in the population density, particularly near the boundary of the domain. We identified a lingering phenomenon on heterogeneous cognitive maps, established the existence and stability of steady states for the logistic version of the model, and used simulations to show that lingering arises even in the presence of population growth.

Our results suggest that once the cognitive map is spatially non-uniform, high learning and moderate forgetting rates can create long-term concentration of individuals near the highest locations of the environment, even when the underlying resource landscape is nearly homogeneous. The cognitive diffusion $\Delta(\gamma(m)u)$ with a cognitive map $m$ effectively reduces movement in regions that are strongly represented in the cognitive map and thus generates lingering around high-valued memory peaks. When logistic growth is included, the learning rate $\alpha$ and the forgetting rate $\mu$ control a trade-off between strong lingering and total population size. For large $\alpha$, we found that an intermediate forgetting rate maximizes the peak density at the resource, whereas a larger $\mu$ maximises the total mass. Very strong memory traps individuals in overcrowded patches, so the regime that maximizes lingering does not maximize population size. Hence the best strategy for population size is fast learning and reasonably fast forgetting, as to be able to adapt to a heterogeneous environment.

For constant motility $\gamma(z)\equiv \gamma$, our model reduces to the classical random diffusive--logistic equation, for which the total mass of the steady state always exceeds the total resource \cite{Lou2006}. Our simulations show that this property can fail under cognitive diffusion: near the strongest-lingering regime, the total mass may drop below that of the landscape $s$. This reduction is also investigated in recent studies of resource-dependent dispersal \cite{Tang2023}, where the fast decaying function $\gamma(z)$ and local information can also lead to a reduction in total population size. In this sense, our work extends the theory of optimal population size from resource-dependent diffusion to memory-dependent diffusion. 

This study has several mathematical limitations. 
For the full model \eqref{eqn:logistic} we proved local well-posedness but we did not establish global-in-time existence of solutions. 
The lack of diffusion in the $m$-equation and the strong coupling in the full system make a direct global existence theory difficult, and this remains an open problem. 
In addition, our rigorous results on the lingering phenomenon concern the model without logistic growth and in the logistic case we rely on numerical simulations. 
For the logistic model, we obtained partial results on lingering and total population size, mainly in the regime of small learning rates $\alpha$ and a complete description for general $(\alpha,\mu)$ is left for future work.

A natural question that arises from this work is an analysis of evolutionary benefits of learning and forgetting. We saw that fast learners and moderate forgetters do best in a heterogeneous environment. Hence next we can look at a competition model where species with different cognitive abilities compete for resources. We expect some interesting insights from such a study, and we leave it for future research.

\color{black}

\section*{Acknowledgments}
The authors thank Hao Wang for discussion at an early stage of this work. TH is supported through a discovery grant of the Natural Science and Engineering Research Council of Canada (NSERC), RGPIN-2023-04269.

\section*{Conflict of Interest Statement}
The authors declare no conflicts of interest.

\section*{Data Availability Statement}
No data was used in this study.

\bibliographystyle{plain}  
\bibliography{references}  

\end{document}